\documentclass[12pt]{amsart}
 
\usepackage{tikz}
\usetikzlibrary{cd}
\usetikzlibrary{decorations.pathreplacing,decorations.markings}

\usepackage{amsfonts,amssymb}

\usepackage{pinlabel,url}
\usepackage[all]{xy}
\usepackage{tikz}
\usepackage{graphicx}
\usepackage{xcolor}
\usepackage{color}
\usepackage[colorlinks,citecolor=blue]{hyperref}

\usepackage[nameinlink]{cleveref}

\theoremstyle{theorem}
\newtheorem{theorem}{Theorem}[section]

\newtheorem{lemma}[theorem]{Lemma}
\newtheorem{proposition}[theorem]{Proposition}
\newtheorem{corollary}[theorem]{Corollary}
\newtheorem{claim}[theorem]{Claim}

\theoremstyle{definition}

\newtheorem*{remark}{Remark}

\newtheorem{definition}[theorem]{Definition}

\numberwithin{equation}{section}

\newcommand{\Vol}{{\rm Vol}}
\newcommand{\Map}{{\rm Map}}

\newcommand{\oct}{{\rm oct}}
\newcommand{\area}{{\rm area}}

\newcommand{\hyp}{{\mathcal{H}}}
\newcommand{\SH}{{\mathcal{SH}}}
\newcommand{\id}{{\rm id}}
\newcommand{\CP}{{\mathcal P}}
\newcommand{\Teich}{\rm{Teich}}

\newcounter{mcomments}

\newcounter{ecomments}

\newcounter{ccomments}

\newcounter{acomments}

\title[A lower bound on volumes of end-periodic mapping tori]{A lower bound on volumes of \\ end-periodic mapping tori}

\author[E. Field, A. Kent, C. Leininger, M. Loving]{Elizabeth Field, Autumn Kent, \\ Christopher Leininger, and Marissa Loving}

\date{\today}

\begin{document}

\maketitle

\begin{abstract}
    We provide a lower bound on the volume of the compactified mapping torus of a strongly irreducible end-periodic homeomorphism $f: S \to S$. This result, together with work of Field, Kim, Leininger, and Loving \cite{EndPeriodic1}, shows that the volume of $\overline M_f$ is comparable to the translation length of $f$ on a connected component of the pants graph $\mathcal P(S)$, extending work of Brock \cite{Brock-mappingtorus-vol} in the finite-type setting on volumes of mapping tori of pseudo-Anosov homeomorphisms.
\end{abstract}

\section{Introduction}

A central theme in the post-geometrization study of $3$-manifolds is to clarify the relationship between geometric and topological features of a manifold.  Fibered hyperbolic $3$-manifolds provide a particularly rich class of examples in this vein, as their topology is completely determined by the monodromy homeomorphism $f \colon S \to S$ of their fiber surface $S$, which realizes the manifold as a mapping torus $M_f$. When $S$ is of finite type, the isotopy class of $f$ is an element of the mapping class group of $S$, and the abundant collection of actions of this group provide a wealth of information about the geometry of $M_f$. 
The current paper is motivated by such a connection due to Brock \cite{Brock-mappingtorus-vol}, who showed that the hyperbolic volume of $M_f$ is comparable to the translation length $\tau(f)$ of $f$  acting on the pants graph $\mathcal P(S)$.
More precisely, he shows that there are constants $K_1$ and $K_2$, depending only on the topology of $S$, such that
\begin{equation}\label{Eq:Brock} K_2 \tau(f) \leq \Vol(M_f) \leq K_1 \tau(f).
\end{equation}

When $S$ has infinite type and $f \colon S \to S$ is a {\em strongly irreducible end-periodic} homeomorphism, earlier work of Field, Kim, Leininger, and Loving \cite{EndPeriodic1} provides an analogous upper bound on the hyperbolic volume of the {\em compactified} mapping torus  with its totally geodesic boundary structure. Namely,
    \[
        \Vol(\overline M_f) \leq C_1 \tau(f)
    \] 
where $C_1$ is in fact simply the volume of a regular ideal octahedron.
We complete the analogy with \eqref{Eq:Brock} by establishing the following lower bound.

\begin{theorem} \label{T:main.theorem} 
For any surface $S$ with finitely many ends, each accumulated by genus, and any strongly irreducible end-periodic homeomorphism $f \colon S \to S$, we have
\[
   C_2 \tau(f) \leq \Vol(\overline M_f),
\]
where the constant $C_2$ depends only on the {\em capacity} of $f$.
\end{theorem}

The {\em capacity} of $f$ is a pair of numbers that records two pieces of \emph{finite} topological data that describe the action of $f$ on $S$---see \Cref{S:complexities}. Since the surface $S$ has infinite topological type, the dependence of $C_2$ on the capacity of $f$ serves as a substitute for the dependence of Brock's $K_2$ on the topology of the finite-type fiber.

There are two pieces of data naturally associated to any $f$-invariant component $\Omega \subset \mathcal P(S)$: 
\begin{enumerate}
    \item the translation length of $f$ on $\Omega$, denoted $\tau_{\Omega}(f)$, and
    \item an induced pants decomposition $P_{\Omega}$ of $\partial \overline M_f$.
\end{enumerate}
See \Cref{sec:bdd-length-inv-comps} for a detailed description.  The following theorem provides a component that (coarsely) optimizes both of these, and ties the action of $f$ on $\mathcal P(S)$ to a bounded length pants decomposition of $\partial \overline M_f$.

\newcommand{\boundedlengththm}{Given $f \colon S \to S$, a strongly irreducible end-periodic homeomorphism, there is a component $\Omega \subset \mathcal P(S)$ and $E >0$ (depending on the capacity of $f$), so that each curve in $P_\Omega \subset \partial \overline M_f$ has length at most $E$, and so that $\tau_\Omega(f) \leq E \tau(f)$.}
\begin{theorem} \label{T:bdd len compts}
   \boundedlengththm 
\end{theorem}

\noindent We expect this theorem to be more generally useful in future analysis of the hyperbolic geometry surrounding depth-one foliations.

\subsection*{Historical notes and future directions}

End-periodic homeomorphisms are an important class of homeomorphisms of infinite-type surfaces, due in large part to their connection with depth-one foliations of 3-manifolds.  Indeed, after collapsing certain trivial product foliation pieces, a co-oriented, depth-one foliation of a $3$-manifold is obtained by gluing together finitely many compactified mapping tori of end-periodic homeomorphisms.  Such foliations (and more generally finite-depth foliations) were studied in detail by Cantwell and Conlon \cite{CantConPB}, and arise in Gabai's analysis \cite{Gabai-Fol1,Gabai-Fol1,Gabai-Fol3} of the Thurston norm \cite{T-Norm}. In \cite{T-Norm}, Thurston also observed that depth-one foliations  occur naturally as limits of foliations by fibers in sequences of cohomology classes limiting projectively to the boundary of the cone on a fibered face of the Thurston norm ball.

In unpublished work, Handel and Miller began a systematic study of end-periodic homeomorphisms using laminations in the spirit of the modern interpretation of Nielsen's approach to the Nielsen--Thurston Classification (see Gilman \cite{Gilman}, Miller \cite{Miller}, Handel--Thurston \cite{HandThur}, and Casson--Bleiler \cite{CasBle}).  Some aspects of this work were described and developed by Fenley in \cite{FenleyThesis1989,Fenley-depth-one}, and more recently expanded upon by Cantwell, Conlon, and Fenley in \cite{CC-book}. The analogy between strongly irreducible end-periodic homeomorphisms and pseudo-Anosov homeomorphisms was further strengthened by work of Patel--Taylor showing that many end-periodic homeomorphisms admit loxodromic actions on various arc and curve graphs of infinite-type surfaces \cite{PatelTaylor}.  

Recent work of Landry, Minsky, and Taylor \cite{LandryMinskyTaylor2023} further studies the behavior of Thurston's depth-one foliations \cite{T-Norm} arising from the boundaries of the fibered faces. In particular, using how the lifts of a first return map act on the boundary circle of a depth one leaf lifted to the universal cover, they relate the invariant laminations for end-periodic homeomorphisms to the  laminations of the pseudo-Anosov flow associated to the fibered face.  Moreover, using veering triangulations \cite{Agol-veering,gueritaud-veering}, they show that any compactified mapping torus appears in the boundary of some fibered face of some fibered $3$--manifold.

Fenley \cite{FenleyCT} provided the first connection between the hyperbolic geometry of a $3$-manifold and its depth-one foliations, proving that when the end-periodic monodromies are irreducible, the depth-one leaves admit Cannon--Thurston maps from the compactified universal covers $\overline{\mathbb H}^2 \to \overline{\mathbb H}^3$.  This is an analogue of Cannon and Thurston's seminal work in the finite-type case (circulated as a preprint for decades before appearing in \cite{CanThu}).  Unlike Cannon and Thurston's map, Fenley's boundary map is not surjective, but rather surjects the limit set of the compactified mapping torus, which is a Sierpinski carpet.  In the course of his arguments, Fenley provides a quasi-isometric comparison between the hyperbolic metric and a (semi)-metric defined by the foliation, which also parallels Cannon and Thurston's approach.

The comparison between the hyperbolic metric and the metric defined by the fibration, as studied by Cannon and Thurston, was greatly elaborated on by Minsky \cite{MinskyTeich} to provide uniform estimates depending on the injectivity radius of the $3$-manifold and the genus of the fiber.  Building on this, and the deep machinery developed by Masur--Minsky \cite{MasurMinsky.1999,MasurMinsky.2000}, Minsky \cite{ELC1} and Brock--Canary--Minsky \cite{ELC2} constructed combinatorial, uniformly biLipschitz models for fibered hyperbolic $3$-manifolds.

Brock's volume estimates \cite{Brock-mappingtorus-vol} above were used to prove his analogous volume estimates in terms of the Weil--Petersson translation length on Teichm\"uller space, but with less control over the constants.  A more direct proof of the Weil--Petersson upper bound, with explicit constants, was proved by Brock--Bromberg \cite{BrockBromberg} and Kojima--McShane \cite{KojimaMcShane} using renormalized volume techniques.

The techniques developed here, and in \cite{EndPeriodic1}, combined with forthcoming work of Bromberg, Kent, and Minsky \cite{BKM}, provides the framework to prove volume estimates for closed hyperbolic $3$-manifolds with depth-one foliations.  One might ultimately hope for a uniform biLipschitz model, but the tools needed to guarantee one seem considerably more difficult.  First steps in this direction are taken by Whitfield in \cite{Whitfield}, where she extends a result of Minsky \cite[Theorem~B]{Minsky.2000} to the infinite-type setting to produce short curves whose lengths are bounded in terms of subsurface projections. From a Teichm\"uller-theoretic perspective, the fact that end-periodic homeomorphisms can be made to act isometrically near the ends suggests the possibility of an action on a Teichm\"uller space with finite translation length, and it is natural to wonder on the relation of such a length to the volume.

Finally, some important motivation for this work comes from the study of big mapping class groups. In particular, \cite[Problem~1.7]{AIM-PL} asks for a characterization of big mapping classes whose mapping tori admit complete hyperbolic metrics. One hope is that better understanding the geometry of end-periodic mapping tori will provide some insights into giving a complete solution to this problem. 

\begin{remark}
We note that this relationship between volumes of hyperbolic 3-manifolds and distance in the pants graph has also been explored in a different setting by Cremaschi--Rodr\'{\i}guez-Migueles--Yarmola \cite{CR-MY} who give upper and lower bounds analogous to those of Brock. 
\end{remark}

\subsection*{Comparison with finite-type case}

We briefly outline here  Brock's strategy for his lower bound \cite{Brock-mappingtorus-vol}, point out the complications that arise when adapting the strategy to our setting, and discuss how we address these challenges.

Brock's proof involves controlling the number and location of bounded length curves in the mapping torus, as each bounded length curve provides a definite contribution to the volume \cite[Lemma~4.8]{Brock-convex-vol}.
To produce bounded length curves, Brock works in the infinite cyclic cover $S \times \mathbb R$ of the mapping torus, and constructs an interpolation between a {\em simplicial hyperbolic surface} \cite{Canary.CoveringTheorem,Bonahon} homotopic to the inclusion of $S$ and the image of this surface under the generator of the deck group.  The deck group acts like $f$ on the $S$ factor, and the interpolation produces a sequence of bounded length pants decompositions, starting with one on the initial surface and ending with its $f$--image in the translate.  He then shows that {\em the number of curves} arising in this sequence provides an upper bound on distance in the pants graph, and hence a bound on the translation length of $f$ \cite[Lemma~4.3]{Brock-convex-vol}.

The interpolation between the two simplicial hyperbolic surfaces can overlap significantly with its translates by the deck transformation, leading to an overestimate in the number of bounded length curves in the mapping torus (as many curves may project to the same curve).  To account for this, Brock first situates a neighborhood of the entire interpolation inside some fixed, but {\em uncontrolled} number $n_0$ of consecutive translates of a fundamental domain for the covering action.  The concatenation of any $j > 0$ consecutive translates of the interpolation produces a path between a pants decomposition and its image under $f^j$, all of whose curves are of bounded length and situated inside $j+n_0$ translates of the fundamental domain.  The number of curves that occur in the sequence of pants decompositions is an upper bound on $j$ times the translation length, and a lower bound on $j+n_0$ times the volume.  Thus, taking $j \to \infty$ and dividing both quantities by $j$, the $n_0$ term disappears, proving the required lower bound on volume in terms of pants translation length.

\begin{figure}[htb!]
    \centering
    \def\svgwidth{6in}
    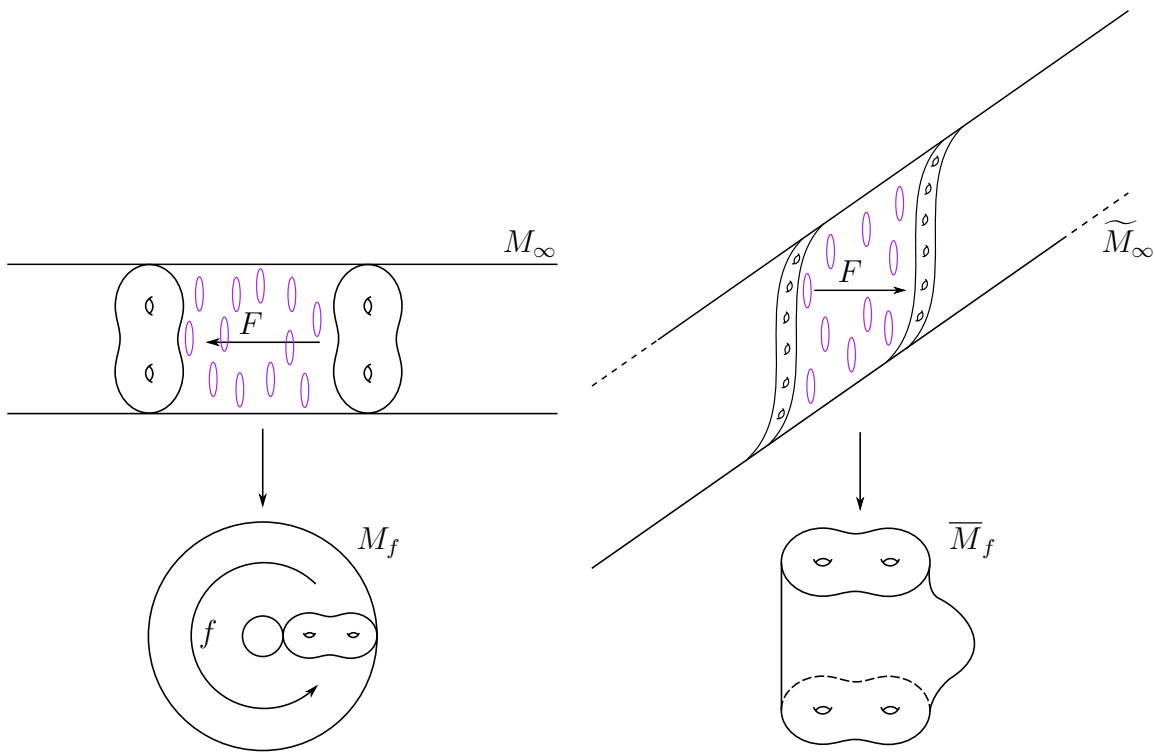
    \caption{In the figure on the right, when we ``interpolate" from left to right through $\widetilde M_{\infty}$ the ``core" of $f$ slides along $S$ as we zip along the bottom boundary and unzip along the top boundary. In both cases, $F$ is the covering transformation given by $F(x, t) = (f(x), t-1)$.} 
    \label{fig:idea-lower-bound}
\end{figure}

Our proof of \Cref{T:main.theorem} involves a similar  strategy. Notably, Brock's lower bound on volume in terms of the number of bounded length curves still serves as the primary mechanism for controlling the volume.  We also carry out many of the arguments in the infinite cyclic cover, though in our situation this cover does not have finitely generated fundamental group. See \Cref{fig:idea-lower-bound} for a cartoon of the differences between the infinite cyclic covers in the two settings.

Several of the ideas in Brock's proof break down in fundamental ways in our setting, and we discuss these in turn. We make extensive use of pleated and simplicial hyperbolic surfaces in the compactified mapping torus with totally geodesic boundary as well as in its infinite cyclic cover.  The infinite-type setting demands some care, but no serious issues arise here.

The first real obstacle we encounter is that the bound on pants distance in terms of the number of curves that appear in all of the pants is not immediately applicable as it relies on the work of Masur and Minsky \cite{MasurMinsky.1999, MasurMinsky.2000}, where the constants depend on the topological type of the surface.  While our pants decompositions contain infinitely many curves, a finite sequence of pants moves takes place on a finite-type subsurface.  Fixing the capacity of $f$ provides an initial bound on the topological type of this subsurface, but even under this condition, two additional issues arise.  First, iteration of the map increases the size of this subsurface linearly in the power, and so the strategy of iterating and taking  limits is not viable.  The second issue is that there is no universal bound on the length of curves in a shortest pants decomposition of a finite-type surface with boundary.

To address the second issue, we construct {\em minimally well-pleated surfaces} in the compactified mapping torus, which send ``as much of the (infinite-type) surface as possible" into the boundary---see \Cref{S:viscera defined} and \Cref{section.pleated}.  Appealing to Basmajian's collar lemma \cite{Basmajian}, the  bound on capacity produces {\em a priori} bounds on the length of the boundary of a minimal core---see \Cref{L:bounding the boundary}.  After passing to a uniform power, we may enlarge the core to ``support" the interpolation through simplicial hyperbolic surfaces, which \textit{does} have a bound on the length of its boundary---see \Cref{S:interpolation}.  This enlarged surface is of bounded topological type, again thanks to bounded capacity, and, in this setting, there is a uniform bound on the length of a bounded length pants decomposition---see \Cref{thm:relative-bers}.

As we cannot iterate and take a limit as Brock does, we address the first issue by essentially gaining control on the ``uncontrolled" constant $n_0$ in Brock's proof.  Interestingly, the feature of end-periodic homeomorphisms that forces the capacity to grow linearly under powers is, along with strong-irreducibility, what comes to the rescue.  Namely, as all of our bounded length curves are homotopic into the enlargement of the core, we can find a uniform power (depending only on the capacity) so that for all higher powers the images of these bounded length curves are {\em not} homotopic into this subsurface (see \Cref{L:no-closed-curves}).  In particular, no two curves in our bounded length set project to the same curve in the mapping torus for this uniform power, and this guarantees that they all contribute to the volume.

\subsection*{Outline of the paper} 
We begin in \Cref{section.prelim} with preliminaries on end-periodic homeomorphisms and their mapping tori, including the definition of capacity.  In \Cref{section.complexity} we establish key {\em topological} features of a strongly irreducible end-periodic homeomorphism acting on a core, and describe how the core sits in the compactified mapping torus. The details for the pleated surface technology we need, and the resulting uniform {\em geometric} features for a strongly irreducible end-periodic homeomorphism are described in \Cref{section.pleated}. The kinds of simplicial hyperbolic surfaces we will use, as well as our applications of these, are described in \Cref{section.interpolation}.  We assemble all the ingredients into the proof of \Cref{T:main.theorem} in \Cref{section.volume}. Finally, in \Cref{sec:bdd-length-inv-comps}, we prove \Cref{T:bdd len compts}.

\section{Preliminaries}\label{section.prelim} 

In this section we set some notation, recall some of the facts that we will need, particularly from \cite{EndPeriodic1}, and define the notion of ``capacity."

\subsection{End-periodic homeomorphisms} \label{subsection:end-periodic}

We restrict our attention to surfaces of infinite-type with finitely many ends, each accumulated by genus, and without boundary. The interested reader can find a more general discussion of end-periodic homeomorphisms in \cite{FenleyCT, Fenley-depth-one, CC-book}.

A homeomorphism of an infinite-type surface $S$ is \textit{end-periodic} if there is an $m > 0$ such that, for each end $E$ of $S$, there is a neighborhood $U_E$ of E so that either
\begin{itemize}
    \item[(i)] $f^m(U_E) \subsetneq U_E$ and the sets $\{f^{nm}(U_E)\}_{n >0}$ form a neighborhood basis of $E$; or
    \item[(ii)] $f^{-m}(U_E) \subsetneq U_E$ and the sets $\{f^{-nm}(U_E)\}_{n > 0}$ form a neighborhood basis of $E$.
\end{itemize}
We say that $E$ is an \textit{attracting end} in the first case, and a \textit{repelling end} in the second.
The neighborhoods $U_E$ are \textit{nesting neighborhoods} of the ends, and when convenient we assume (as we may) that we have chosen disjoint nesting neighborhoods for distinct ends.
We denote the union of the neighborhoods of the attracting ends $U_+$ and write $U_-$ for the union of the neighborhoods of the repelling ends.  If $f^{\pm 1}(U_\pm) \subset U_\pm$, and $\partial \overline U_\pm$ is a union of simple closed curves, then we say that $U_\pm$ are {\em tight nesting neighborhoods}.  Every end-periodic homeomorphism admits tight nesting neighborhoods. For instance, the good nesting neighborhoods from \cite{EndPeriodic1} are a particular example of tight nesting neighborhoods, with the additional assumption that each component of $\overline U_\pm$ has a single boundary component.

A compact subsurface $Y \subset S$ is a {\em core} for $f$ if $S-Y$ is a disjoint union of tight nesting neighborhoods $U_+$ and $U_-$. Given a core $Y$, define the {\em junctures} $\partial_+Y$ and $\partial_-Y$ to be the boundary components meeting $U_+$ and $U_-$, respectively.  Note that there are infinitely many choices of cores for $f$ (as a given core can always be enlarged).

Given a core $Y$ for $f$, a hyperbolic metric on $S$
for which $f|_{U_+} \colon  U_+ \to U_+$ and $f^{-1}|_{U_-} \colon U_- \to U_-$ are isometric embeddings is said to be {\em compatible with $Y$}.  Adjusting $f$ by an isotopy if necessary, there are always metrics which are compatible with a given core $Y$, see \cite{Fenley-depth-one}.  

\begin{remark}
    We will adjust our homeomorphism $f$ by an isotopy several more times in what follows.  While we could impose all the conditions we will need at the outset, each will require additional discussion and set-up, and so it seems natural to impose each condition as they arise.  It will be evident that each additional condition does not contradict the previous ones, but we will indicate any subtleties as they arise.
\end{remark}

Define 
\[ 
    \mathcal U_+ = 
        \bigcup_{n \geq 0} f^{- n}(U_+) 
            \quad \mbox{ and } \quad 
        \mathcal U_- = \bigcup_{n \geq 0} f^{n}(U_-),
\]
which are the \textit{positive} and \textit{negative escaping sets for $f$}, respectively. We note that any choice of nesting neighborhoods will give rise to the same escaping sets $\mathcal U_\pm$ (depending only on the homeomorphism $f$). 
With these assumptions, the restrictions $\langle f\rangle|_{\mathcal U_\pm}$ act cocompactly on $\mathcal U_\pm$ with quotients $S_\pm = \mathcal U_\pm/\langle f \rangle$, which may be disconnected, see, \textit{e.g.} \cite[Lemma 2.4]{EndPeriodic1}. 

In the following, \textit{curve} and \textit{line} refer to proper homotopy classes of essential simple closed curves and essential properly embedded lines, respectively. A curve $\alpha$ is called \emph{reducing} with respect to an end-periodic homeomorphism $f$ if there exists $m, n \in \mathbb Z$ with $m < n$ and such that $f^n(\alpha)$ (has a representative that) is contained in a nesting neighborhood of an attracting end and $f^m(\alpha)$ (has a representative that) is contained in a nesting neighborhood of a repelling end. 

\begin{definition}[Strong irreducibility]\label{Irreducibility/Strong irreducibility} An end-periodic homeomorphism, ${f:S\to S}$, is \emph{strongly irreducible} if it has no \emph{periodic curves}, no \emph{periodic lines}, and no \emph{reducing curves}. \end{definition}

\subsection{Mapping tori and their compactifications} \label{subsection:mapping-tori}

We define a partial compactification of $S \times (-\infty, \infty)$ inside $S \times [-\infty, \infty]$ by 

\[ 
    \widetilde M_\infty = \{ (x,t) \in S \times [-\infty,\infty] \mid x \in \mathcal U_\pm \mbox{ if } t = \pm \infty \},
\]
and define $F \colon \widetilde M_\infty \to \widetilde M_\infty$ by $F(x,t) = (f(x),t-1)$, where $\pm \infty - 1 = \pm \infty$.  
The group $\langle F \rangle$ acts properly discontinuously and cocompactly on $\widetilde M_\infty$, see, \textit{e.g.} \cite[Lemma 3.2]{EndPeriodic1}. The quotient $p \colon \widetilde M_\infty \to \overline M_f = \widetilde M_\infty/\langle F \rangle$ is a compact manifold with boundary naturally homeomorphic to $S_- \cup S_+$ and whose interior is the mapping torus $M_f$ of $f$ which we call the {\em compactified mapping torus}. The manifold $\overline M_f$ was first defined by Fenley \cite{Fenley-depth-one}. It is particularly nice when $f$ is strongly irreducible. 

\begin{theorem}{\cite[Proposition 3.1]{EndPeriodic1}}
Let $f: S\to S$ be a strongly irreducible, end-periodic homeomorphism of a surface with finitely many ends, all accumulated by genus. Then $\overline M_f$, is a compact, irreducible, atoroidal, acylindrical $3$-manifold, with incompressible boundary.
\end{theorem}

Together with Thurston's Geometrization Theorem for Haken manifolds and Mostow Rigidity \cite{Thurston.Geometrization.1, mcmullen1992riemann, Morgan}, the above result implies the following theorem.

\begin{theorem} \label{thm:convex-hyperbolic}
If $f: S\to S$ is a strongly irreducible, end-periodic homeomorphism of a surface with finitely many ends, all accumulated by genus, then $\overline{M}_f$ admits a convex hyperbolic metric $\sigma_0$ with totally geodesic boundary, which is unique up to isometry.
\end{theorem}

Whenever discussing metric properties of $\overline M_f$, we will assume it is equipped with the convex hyperbolic metric $\sigma_0$, and may simply refer to it as {\em the} hyperbolic metric on $\overline M_f$ (due to the uniqueness statement), without specific reference to its name. The metric  $\sigma_0$ pulls back to a complete hyperbolic metric on $\widetilde M_\infty$ for which $\mathcal U_\pm \times \{\pm \infty\}$ is totally geodesic.  

Given a core $Y$ with tight nesting neighborhoods $U_\pm$, we choose a hyperbolic metric $\mu$ on $S$ so that the ``inclusions'' $U_\pm \to U_\pm \times \{\pm \infty\}$ into $\partial \widetilde{M}_\infty$ are isometric embeddings.  Then, after adjusting $f$ by an isotopy on $U_\pm$ if necessary, $\mu$ is compatible with $Y$.  This is possible since the inclusion $U_\pm \to U_\pm \times \{\pm \infty\}$ conjugates the restriction of $f^{\pm 1}$ to the restriction of $F^{\pm 1}$, which acts isometrically.  We say that such a metric is {\em induced by the metric on $\widetilde M_\infty$}.

\subsection{Euler characteristic, complexity, and capacity} \label{S:complexities}

Given a compact surface $Z$ of genus $g$ with $n$ boundary components, there are two measures of the ``size" of $Z$; $\chi(Z)$, the Euler characteristic, and the complexity $\xi(Z)= 3g-3+n$. We note that $\xi(Z)$, when positive, is the maximal number of essential, pairwise disjoint, pairwise non-isotopic simple closed curves on $Z$, \textit{i.e.}~the number the of curves in a pants decomposition of $Z$.  Since $\chi(Z) = 2-2g-n$, we have the following elementary fact for all $Z$ with $\xi(Z) \geq 0$,
\[ |\chi(Z)|-1 \leq \xi(Z) \leq \tfrac{3}2 |\chi(Z)|.\]
For $Z$ closed (and genus at least $1$) the second inequality is an equality.  We extend both of these quantities to disconnected surfaces, additively over the components (which is natural for the Euler characteristic), and observe that when all components have $\xi \geq 0$ (the only case of interest for us), the inequality on the right still holds.  We use all of this in what follows without explicit mention.

Given an end-periodic homeomorphism $f \colon S\to S$, we define the \emph{core characteristic of $f$} to be
\[ \chi(f) = \max_{Y \subset S} \chi(Y), \]
where the maximum is taken over all cores $Y \subset S$ for $f$. Informally, a core is a subsurface where curves from the repelling end get ``hung up" under forward iteration of $f$ (or where curves from the attracting end get hung up under backward iteration). Thus, $\chi(f)$ measures the minimal size of the subsurface where that behavior occurs. Any core $Y$ with $\chi(Y) = \chi(f)$ will be called a {\em minimal core}.  
\begin{remark}
Note that it is always possible to choose a core $Y$ so that each component of $U_+$ or $U_-$ meets $Y$ in a single simple closed curve (see \cite[Corollary~2.5]{EndPeriodic1} and the discussion preceding it).  In this case, connectivity of $S$ implies connectivity of $Y$. However, cores need not be connected, as the example in \Cref{F:weird cores} illustrates.

\begin{figure}[htb!]
\centering
\def\svgwidth{4in}
\begingroup%
  \makeatletter%
  \providecommand\color[2][]{%
    \errmessage{(Inkscape) Color is used for the text in Inkscape, but the package 'color.sty' is not loaded}%
    \renewcommand\color[2][]{}%
  }%
  \providecommand\transparent[1]{%
    \errmessage{(Inkscape) Transparency is used (non-zero) for the text in Inkscape, but the package 'transparent.sty' is not loaded}%
    \renewcommand\transparent[1]{}%
  }%
  \providecommand\rotatebox[2]{#2}%
  \newcommand*\fsize{\dimexpr\f@size pt\relax}%
  \newcommand*\lineheight[1]{\fontsize{\fsize}{#1\fsize}\selectfont}%
  \ifx\svgwidth\undefined%
    \setlength{\unitlength}{888.37648345bp}%
    \ifx\svgscale\undefined%
      \relax%
    \else%
      \setlength{\unitlength}{\unitlength * \real{\svgscale}}%
    \fi%
  \else%
    \setlength{\unitlength}{\svgwidth}%
  \fi%
  \global\let\svgwidth\undefined%
  \global\let\svgscale\undefined%
  \makeatother%
  \begin{picture}(1,0.54937404)%
    \lineheight{1}%
    \setlength\tabcolsep{0pt}%
    \put(0,0){\includegraphics[width=\unitlength,page=1]{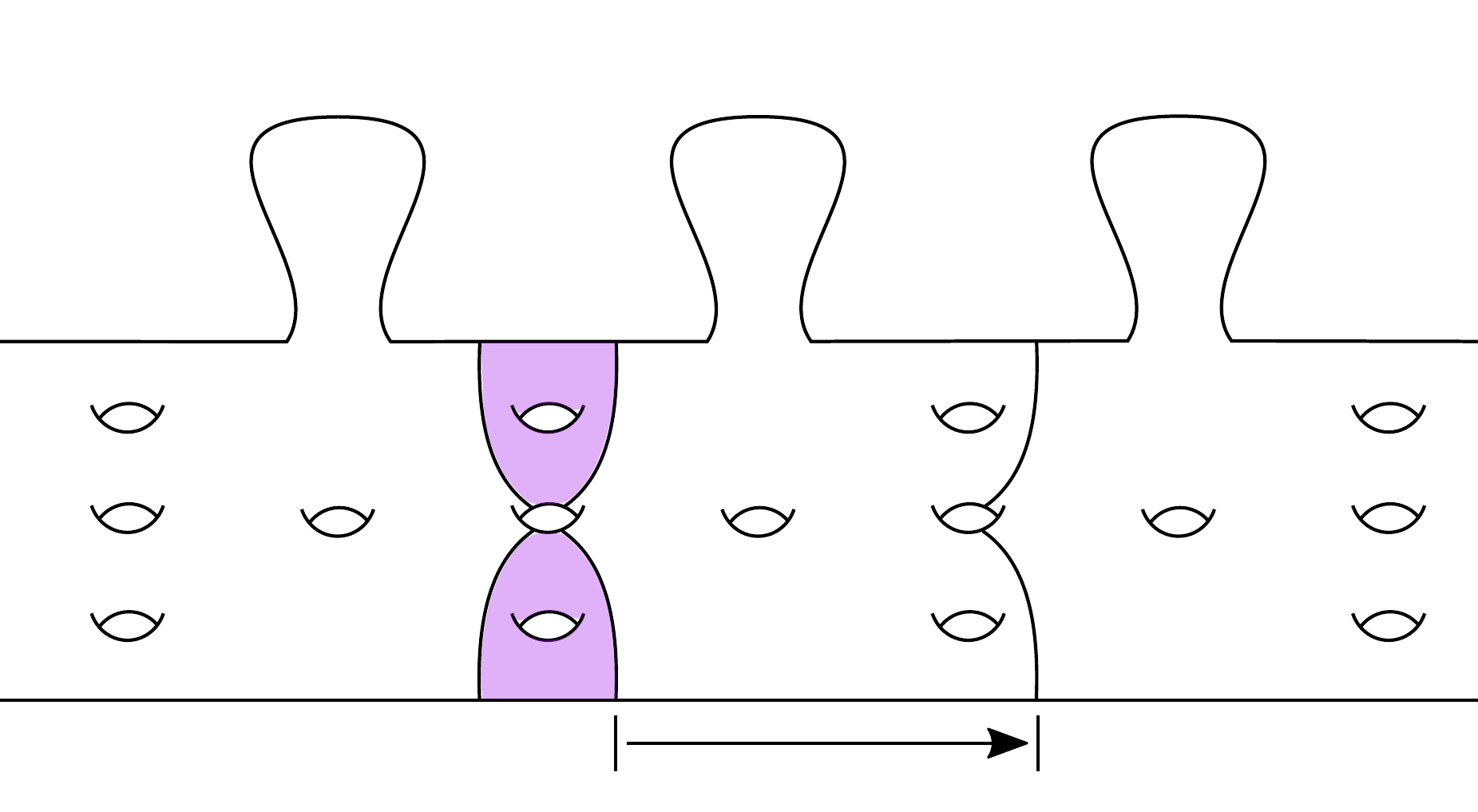}}%
    \put(0.55951853,0.00595251){\color[rgb]{0,0,0}\makebox(0,0)[t]{\lineheight{1.25}\smash{\begin{tabular}[t]{c}$\rho$\end{tabular}}}}%
    \put(0,0){\includegraphics[width=\unitlength,page=2]{disconnected-core.pdf}}%
    \put(0.51217755,0.52338746){\color[rgb]{0,0,0}\makebox(0,0)[t]{\lineheight{1.25}\smash{\begin{tabular}[t]{c}$h$\end{tabular}}}}%
    \put(0,0){\includegraphics[width=\unitlength,page=3]{disconnected-core.pdf}}%
  \end{picture}%
\endgroup%

\caption{The homeomorphism $\rho$ generates a covering action with quotient a genus $6$ surface, and $h$ is a partial pseudo Anosov supported on the subsurface which is the shaded (purple and gray) subsurface with genus $6$ and four boundary components. The disconnected shaded subsurface (purple in the figure) is a core for $f = h \circ \rho$, regardless of what $h$ is.  Note that one of the components of this core has a single boundary component that faces $U_+$, and thus no component that faces $U_-$.  We can remove this component, and what remains is still a core for $f$.  For $h$ ``sufficiently complicated'', $f = h \circ \rho$ will be strongly irreducible.}
\label{F:weird cores}
\end{figure}
\end{remark}

Given a core $Y$ for a strongly irreducible end-periodic homeomorphism $f \colon S \to S$, some components of $Y$ may be disjoint from either $\partial_-Y$ or $\partial_+Y$.  We call such a component {\em imbalanced}, and say that $Y$ is {\em balanced} if there are no imbalanced components.  The next lemma says that these imbalanced components can always be safely ignored.  While all the arguments in this paper hold regardless of whether or not there are imbalanced components, it can be helpful in developing intuition to assume there are none.

\begin{lemma} \label{L:balanced subcore}
If $f \colon S\to S$ is a strongly irreducible end-periodic homeomorphism and $Y$ is a core for $f$, then there is a subsurface $Y' \subset Y$ which is balanced.  In particular, any minimal core is balanced.
\end{lemma}
\begin{proof}
Let $Y$ be a core for $f$ and suppose there is an imbalanced component disjoint from $\partial_-Y$.  We first show that we can remove at least one such component to obtain a new core for $f$.  To that end, let $Y_0 \subset Y$ be the union of all imbalanced components with $\partial_-Y_0 = \emptyset$. Since $f(U_+) \subset U_+$, and the components of $\{f^k(U_+)\}_{k>0}$ determine a neighborhood basis for the attracting ends, it follows that $f^{-n}(\partial_+ Y)$ has no transverse intersections with $\partial_+Y$, for all $n \geq 1$.  
Now set $Y_1$ to be the (possibly empty) intersection $Y_1 = Y \cap f(Y_0)$.  Observe that $f^{-1}(Y_1) = f^{-1}(Y) \cap Y_0$, is a (possibly empty) subsurface of $Y_0$ whose boundary components are either components of $\partial_+Y_0$ or else are contained in the interior of $Y_0$ and hence $\partial f^{-1}(Y_1) \subset f^{-1}(\partial_+ Y)$.  Consequently, $Y_1$ is a union of components of $Y_0$.  

We note that $Y_1$ is a proper subsurface of $Y_0$ since otherwise, $f^{-1}$ would simply permute the boundary components of $Y_0$, creating a periodic curve, contradicting the strong irreducibility of $f$. Continuing in a similar way, we see that $Y_2 = Y \cap f(Y_1)$ is a union of components of $Y_1$. In fact, this subsurface $Y_2$ must be a proper subsurface of $Y_1$ by a similar argument as above. Continuing inductively, we find a nested sequence
\[ Y_0 \supset Y_1 \supset Y_2 \supset \cdots \]
defined by $Y_{j+1} = Y \cap f(Y_j)$, for all $j \geq 0$. Furthermore, $Y_j \neq Y_{j+1}$ if $Y_j \neq \emptyset$.

As the areas of the subsurfaces in this sequence always decrease by a multiple of $\pi$, there is some smallest $n \geq 1$ so that $Y_n = \emptyset$ (and then $Y_j = \emptyset$ for all $j \geq n$).  It follows that $Y_{n-1} \neq \emptyset$ is a union of components of $Y_0$ and $f(Y_{n-1}) \cap Y = \emptyset$.  Since $f(\partial_+Y) \subset \overline U_+$, it follows that $f(Y_{n-1}) \subset U_+$.  Therefore, $W = Y-Y_{n-1}$ must also be a core. Indeed,
\[ S - W = U_- \sqcup (U_+ \cup Y_{n-1}), \]
and since $f(Y_{n-1}) \subset U_+$, we have
\[ f^{-1}(U_-) \subset U_- \mbox{ and } f(U_+ \cup Y_{n-1}) \subset U_+ \subset U_+ \cup Y_{n-1},\]
meaning that $U_-$ and $U_+\cup Y_{n-1}$ are tight nesting neighborhoods.

We thus have a new core and have reduced the number of imbalanced components.  Repeating this procedure finitely many times we can remove all imbalanced components with non-empty positive boundary.  Likewise, repeating for the union $Y_0'$ of imbalanced components with $\partial_+Y_0' = \emptyset$, and replacing $f$ with $f^{-1}$ in the arguments above, we can remove all imbalanced components.
\end{proof}

\Cref{L:balanced subcore} together with \Cref{L:intution matched} ensures that we can always select a minimal core which produces nesting neighborhoods of the ends each with a single boundary component. This is helpful to record as it justifies the intuition that we have used throughout the paper. 

\begin{lemma} \label{L:intution matched}
If $Y \subset S$ is a balanced core for a strongly irreducible, end-periodic homeomorphism and $\mathcal U_0 \subset \mathcal U_\pm$ is a component defined by a component $U_0 \subset S-Y$, then $\partial \overline U_0$ separates $\mathcal U_0$ into two components, each a neighborhood of an end of $\mathcal U_0$.
\end{lemma}

\begin{proof} Suppose $\mathcal U_0 \subset \mathcal U_+$, for concreteness (the other case follows by replacing $f$ with $f^{-1}$).
By construction, $U_0 \subset \mathcal U_0 - \partial U_0$ is an unbounded component which is a neighborhood of an end of $\mathcal U_0$.  Since $\mathcal U_0$ has two ends, there is at least one other component $U_0' \subset \mathcal U_0 - \partial U_0$, which is necessarily unbounded.  We must show that $\mathcal U_0 - \partial U_0 = U_0 \cup U_0'$.

Suppose there is another component, $U_1 \subset \mathcal U_0 - \partial U_0$, different from $U_0,U_0'$.  Observe that $U_1$ is the interior of a compact subsurface $\overline U_1$ with
\[ \partial \overline U_1 \subset \partial U_0 \subset \partial_+Y.\]  
We claim that $\partial Y - \partial U_0$ cannot intersect $\overline U_1$.  Indeed, it is disjoint from $\partial U_0$, so if it intersected $\overline U_1$, it would necessarily be contained in it. But every component $\alpha \subset \partial Y - \partial U_0$ faces a neighborhood of another end of $S$ (different than $U_0$), and so $\alpha$ is contained in a component $\mathcal U_\pm$ {\em different} than $\mathcal U_0$.  The claim implies $\overline U_1$ is a component of $Y$ with $\partial \overline U_1 \subset \partial  U_0 \subset \partial_+Y$; thus, $\overline U_1$ is an imbalanced component, which is a contradiction.  The lemma follows.
\end{proof}

For any end-periodic homeomorphism $f \colon S \to S$, we also define the \emph{end complexity of $f$} to be
\[ 
    \xi(f) = \xi(\partial \overline M_f).
\]
Note that in \cite{EndPeriodic1}, the right-hand side is the notation for this complexity, but because it will appear often later, we have adopted this short-hand.

Given any core $Y$ for $f$ with tight nesting neighborhoods $U_\pm$, set
\[ \Delta_\pm = \overline{U_\pm - f^{\pm 1}U_\pm},\] which are compact subsurfaces in $\overline U_\pm$.   We note that $\Delta_\pm \subset \mathcal U_\pm$ serves as a fundamental domain for the restricted action of $\langle f\rangle|_{{\mathcal U}_\pm}$, therefore $\chi(\Delta_\pm) = \chi(S_\pm)$ (see the proof of \cite[Corollary~2.5]{EndPeriodic1}).  Therefore,
\[ |\chi(\Delta_\pm)| = |\chi(S_\pm)|= \tfrac23\xi(S_\pm),\]
since $S_\pm$ are closed.
From the same corollary, $\xi(S_+) = \xi(S_-)$, and thus
\[ 
    \xi(f) = \xi(\partial \overline M_f) = \xi(S_+) +\xi(S_-) = 2\xi(S_+)= 3|\chi(\Delta_\pm)|.  
\]
Thus, $\xi(f)$ can be thought of as measuring the amount of translation of $f$ on the ends of $S$; an alternative perspective on this is explained in \cite[Corollary~2.8]{EndPeriodic1} which connects to work of Aramayona-Patel-Vlamis \cite{aramayona2020first}.

Taken together, $\chi(f)$ and $\xi(f)$ provide a measure of the topological complexity of $f$. More precisely, $f$ acts by ``translating" from the negative ends of $S$ into the positive ends with some amount of ``mixing" happening in a compact subsurface: $\xi(f)$ measures how much $f$ translates by and $\chi(f)$ measures how large of a subsurface the mixing takes place on. These two quantities thus serve as a substitute for the genus, Euler characteristic, or complexity of a finite-type surface.

\begin{definition}
For any end-periodic homeomorphism $f$, we call the pair $(\chi(f),\xi(f))$, the {\em capacity} of $f$.
\end{definition}

We also note that since $\xi(f) = \frac32|\chi(\partial \overline M_f)|$ and since for any $n > 0$, $\overline M_{f^n}$ is an $n$--fold cover of $\overline M_f$, we have \[ \xi(f^n) = \tfrac32|\chi(\partial \overline M_{f^n})| = n\tfrac{3}2|\chi(\partial \overline M_f)| = n\xi(f).\]
Thus raising to powers increases end-complexity in a predictable way.  On the other hand, a core for $f$ is also a core for $f^n$, and hence core characteristic is non-decreasing under raising to powers.

\subsection{The pants graph and Bers pants decompositions}

A {\em pants decomposition} on $S$ is a multicurve $P$ in $S$ such that $S - P$ is a collection of three-holed spheres (\textit{i.e.}~pairs of pants). An {\em elementary move} on a pants decomposition $P$ replaces a single curve in $P$ with a different one intersecting it a minimal number of times, producing a new pants decomposition $P'$. There are two types of elementary moves corresponding to whether the complexity one subsurface in which the elementary move takes place is a one-holed torus or a four-holed sphere. This is illustrated in \Cref{fig:pants-moves}. 

\begin{figure}[htb]
    \centering
    \includegraphics[width = 7 cm]{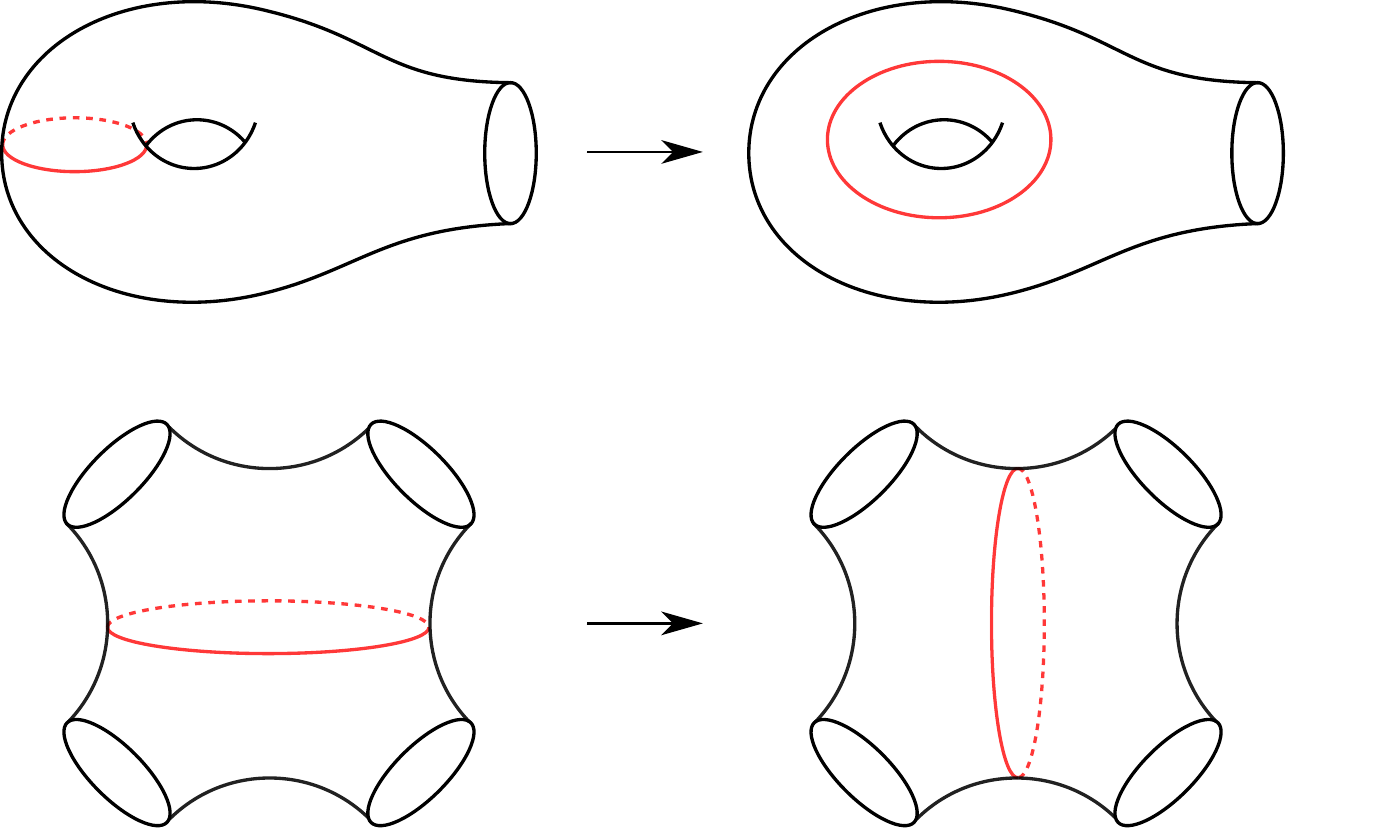}
    \caption{Elementary moves on pants decompositions.}
    \label{fig:pants-moves}
\end{figure}

\begin{definition} \label{def:pants-graph}
The {\em pants graph}, $\mathcal P(S)$, is the graph whose vertices are (isotopy classes of) pants decompositions on $S$, with edges between pants decompositions that differ by an elementary move. 
\end{definition} 

There is a path metric on (the components of) $\mathcal P(S)$ with respect to which the action of $\Map(S)$ on $\mathcal P(S)$ is isometric. It is defined as follows: an edge corresponding to an elementary move that occurs on a one-holed torus has length $1$, and an edge corresponding to an elementary move that occurs on a four-holed sphere has length $2$. 

Brock proved in \cite{Brock-mappingtorus-vol} that for finite-type surfaces $Z$, $\mathcal P(Z)$ is quasi-isometric to the Teichm\"uller space of $Z$, $\Teich(Z)$ equipped with the Weil--Petersson metric. Although this correspondence no longer holds in the infinite-type setting, as we shall see, the pants graph still encodes important geometric data. 

\begin{definition}
Given any end-periodic homeomorphism $f \colon S \to S$, we define the \textit{asymptotic translation distance} of $f$ on $\mathcal{P}(S)$ to be \[\tau(f) = \inf_{P\in\mathcal{P}(S)} \liminf_{n \to \infty} \frac{d(P, f^n(P))}n,\]
where this infimum is over all pants decompositions $P \in \mathcal P(S)$. Observe that $\tau(f^n) = n \tau(f)$ for all $n > 0$.
\end{definition} 

Note that $\mathcal P(S)$ is necessarily disconnected when $S$ is of infinite type (see Branman \cite{Branman} for more on the pants graphs of infinite-type surfaces).  In particular, for certain $P \in \mathcal P(S)$, $d(P,f^n(P))$ is infinite for all $n > 0$.  Consequently, the infimum is effectively being taken over the union of connected components for which this distance is finite for some, hence infinitely many, $n>0$. In \cite{EndPeriodic1}, such $P$ were called $f$--asymptotic pants decompositions as these pants decompositions are $f$--invariant on neighborhoods of the ends of $S$. 

Throughout the paper it will be necessary to produce pants decompositions of bounded length. Bers proved  that a closed hyperbolic surface of genus $g$ admits a pants decomposition for which all the curves have length bounded by a constant depending only on the genus \cite{BersP1,BersP2}.  We will need a relative version of Bers result for surfaces with boundary, a short proof of a very concrete version of which was recently given by Parlier \cite{Parlier2023}.

\begin{theorem}{\cite[Theorem~1.1]{Parlier2023}}\label{thm:relative-bers}
    Let $X$ be a hyperbolic surface, possibly with geodesic boundary, and of finite area. Then $X$ admits a pants decomposition where each curve is of length at most \[L_B = \max{\{\ell(\partial X), \area(X)\}}.\]
\end{theorem}

\section{Cores and topology}
\label{section.complexity}

Let $f: S \to S$ be a strongly irreducible end-periodic homeomorphism.  In this section, we will prove some additional topological information about $f$, a core $Y \subset S$ for $f$, and features of $Y$ that are reflected in $\overline M_f$.

\subsection{Uniform power bounds}

After applying a sufficiently large power of $f$, some part of any curve $\alpha\in Y$ must leave $Y$ ({\em c.f.}~\cite[Theorem~2.7(iii)]{Fenley-depth-one} and \cite[Lemma~2.1]{LandryMinskyTaylor2023}).  We will need the following strengthening of this fact which provides a uniform power for which that behavior occurs.

\begin{lemma} \label{L:no-closed-curves}
Given a core $Y\subset S$ for a strongly irreducible end-periodic $f \colon S \to S$, then for all $k \geq 2\xi(Y)$ there are no closed, essential curves contained in $f^k(Y) \cap Y$.
\end{lemma}

\begin{proof}
We prove the equivalent statement that there are no essential curves in $f^{-k}(Y) \cap Y$, since this introduces fewer total inverses in the proof.

We first make a definition and record an observation. 
We say that an essential, possibly disconnected, subsurface $Z \subset S$ is {\em lean} if no components are pants and no two annular components are homotopic.
Given a lean subsurface $Z \subset S$, define 
\[ 
    \zeta(Z) = (\xi(Z_0),y), 
\] 
where $Z_0 \subset Z$ is the union of all non-annular components of $Z$, $\xi(Z_0)$ is the complexity of $Z_0$ (as defined in \Cref{section.prelim}), and $y$ is the number of annular components of $Z$.  Observe that $\xi(Z_0) + y$ is the number of pairwise disjoint, non-parallel curves in $Z$ (here the core curve of an annulus is considered an essential curve in the annulus).
We consider such pairs as elements of $\mathbb Z^2$ with the dictionary order.

Suppose that $Z \subset Z'$ are lean subsurfaces of $S$ (where we assume that any annular component of $Z$ is either contained in an annular component of $Z'$, or in a non-annular component of $Z'$ in which it is non-peripheral). Then $\zeta(Z) \leq \zeta(Z')$, with equality if and only if there is a homeomorphism $h \colon S \to S$ isotopic to the identity so that $h(Z) = Z'$.  Moreover, observe that if we write $\zeta(Z) = (\xi,y)$ and $\zeta(Z') = (\xi',y')$ and if $Z'$ is the union of $Z$ with the regular neighborhood of a multicurve, then $\zeta(Z) \leq \zeta(Z')$ implies that $\xi+y \leq \xi'+y'$. 

Now, suppose to the contrary that there is a curve $\alpha \subset f^{-k}(Y) \cap Y$ for some $k \geq 2 \xi(Y)$.  Then, $f^k(\alpha) \subset Y$, and since $S-Y = U_+ \cup U_-$, it follows that $f^j(\alpha) \subset Y$ for all $j=0,\ldots,k$, since $f(U_+) \subset U_+$. 

For each $j=0,\ldots,k$, let $Z_j \subset Y$ be the smallest lean subsurface filled by
\[ 
    \bigcup_{i=0}^j f^i(\alpha),
\]
and write $\zeta(Z_j) = (\xi_j,y_j)$.
Then $\zeta(Z_j) \leq \zeta(Z_{j+1}) \leq \zeta(Y) = (\xi(Y),0)$ and $\xi_j+y_j \leq \xi_{j+1}+y_{j+1} \leq \xi(Y)$ for all $0\leq j < k$.  Observe that for any $j = 0, \ldots, k$, either $\zeta(Z_j) = \zeta(Z_{j+1})$, or one of the following strict inequalities must hold:
\[ 
    \xi_j < \xi_{j+1} \quad \mbox{ or } \quad  \xi_j+y_j < \xi_{j+1} + y_{j+1}. 
\]
This implies that the sequence of $L^1$--norms of the pairs $\{(\xi_j,\xi_j+y_j)\}_{j=0}^k$ is a non-decreasing sequence of integers from $1$ to $2\xi_k + y_k \leq 2\xi(Y)$. 
But since $k \geq 2\xi(Y)$, there must be consecutive pairs $(\xi_j,\xi_j+y_j)$ and $(\xi_{j+1},\xi_{j+1}+y_{j+1})$ whose $L^1$--norms are equal, and so, for this $j$, we have $\zeta(Z_j) = \zeta(Z_{j+1})$.  In particular, there is a homeomorphism $h \colon S\to S$ isotopic to the identity so that $h(Z_j) = Z_{j+1}$.

Since $f(Z_j) \subset Z_{j+1}$ and $\zeta(f(Z_j)) = \zeta(Z_j)$, there is a homeomorphism $h' \colon S \to S$ isotopic to the identity so that $h'(f(Z_j)) = Z_{j+1} = h(Z_j)$.
Rewriting this, we have $h^{-1}h'f(Z_j) = Z_j$.  
But $h^{-1}h'f$ is isotopic to $f$, and we conclude that $f$ preserves $Z_j$, and hence $\partial Z_j$, up to isotopy, contradicting the strong irreducibility of $f$.
\end{proof}

\subsection{Cores in the compactified mapping torus} \label{S:viscera defined}

The suspension flow on $M_f$ can be reparameterized and extended to a local flow $(\psi_s)$ on $\overline M_f$.  
Fixing a core $Y$ for $f$ with the associated (tight) nesting neighborhoods, $U_+$ and $U_-$, we can define a homotopy of the inclusion $S \to \overline M_f$ along the flowlines, by flowing $U_+$ and $U_-$ forward and backward, respectively, until they meet $\partial \overline M_f$.  We do this, carrying along a small neighborhood of $\partial Y$ in $Y$, but keeping the rest of $Y$ fixed.  If we let $h_t \colon S \to \overline M_f$, $t \in [0,1]$ denote the homotopy, then we can assume that $h_t$ is injective for all $t \in [0,1)$.  One way to think about this construction is via {\em spiraling neighborhoods} of the boundary; see {\em e.g.}~\cite[\S4]{Fenley-depth-one},\cite[\S3.1]{LandryMinskyTaylor2023}.
We write $Y_1 = h_1(Y) \subset \overline M_f$, and call $Y_1$ the {\em $Y$--viscera}.  It is convenient to think of $h_1(S)$ as a branched surface in $\overline M_f$, transverse to $(\psi_s)$.

Since the first return map of $(\psi_s)$ to $S$ is $f$, the result is a map of $S$ into $\overline M_f$ which is embedded on the interior of $Y$, and for which $U_+$ and $U_-$ map onto $\partial \overline M_f \cong S_+ \sqcup S_-$.  After adjusting (precomposing) $f$ by an isotopy supported in a small neighborhood of $\partial Y$, we may assume that $f(\overline U_+) \subset U_+$ and $f^{-1}(\overline U_-) \subset U_-$.  Having done so, $Y_1 \subset \overline M_f$ is then properly embedded. This assumption is only made to carry out the proofs in this section.  Without it, we can carry out the above homotopy, and the $Y$--viscera will be embedded on the interior, but not necessarily on the boundary.  In particular, we only make this assumption in this section.

In the remainder of this section, we may need to adjust $f$ by an isotopy which is the identity outside of $Y$.  This does not affect the homeomorphism type of the pair $(\overline M_f,Y_1)$, and we use the same notation to denote the new pair.

A \emph{boundary-compressing disk} for $Y_1 \subset \overline M_f$ (or more generally, for any properly embedded surface) is an embedded disk $D \subset \overline M_f$ such that $\partial D = \alpha \cup \beta$, where $\alpha$ is a properly embedded essential arc in $Y_1$, $\beta \subset \partial \overline M_f$, $\alpha$ and $\beta$ intersect precisely in their endpoints, and the interior of $D$ is disjoint from $Y_1 \cup \partial \overline M_f$.  We can perform a homotopy of $Y_1$, rel $\partial Y_1$, ``pushing across $D$" so that a neighborhood of $\alpha$ in $Y_1$ is mapped into $\partial \overline M_f$. See \Cref{F:boundary-compressing-disk}. 

\begin{figure}[htb!]
\def\svgwidth{5in}
    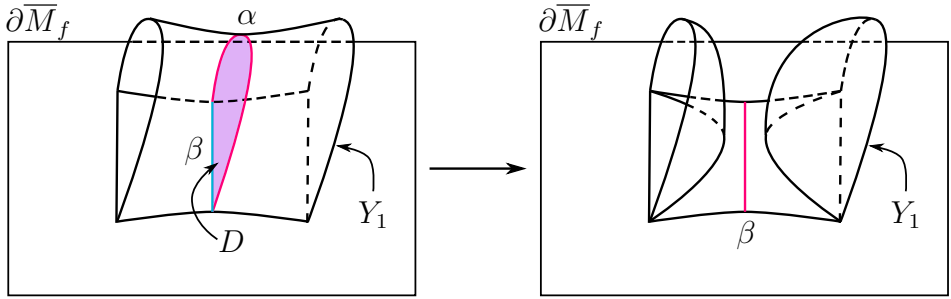
\caption{A boundary-compressing disk $D$ (shaded purple) cobounded by an arc $\alpha \subset Y_1$ and an arc $\beta \subset \partial \overline M_f$. On the right of the figure is $Y_1$ after $D$ is compressed.} \label{F:boundary-compressing-disk}
\end{figure}

A {\em flow-compressing disk} for $Y_1$ is a boundary-compressing disk $D \subset \overline M_f$ for $Y_1$ that is foliated by flowlines transverse to $\alpha$ and $\beta$.  Given a flow-compressing disk $D$, if the arc $\beta$ lies in $S_+ \subset \partial \overline M_f$, then $\alpha' = h_1^{-1}(\alpha)$ is an arc in $Y$, and $\beta = h_1(f(\alpha'))$.  Since $h_1(f(\alpha')) \subset S_+$, it follows that $f(\alpha') \subset U_+$. Consequently, a flow-compressing disk is really determined by a properly embedded essential arc $\alpha' \subset Y$ that is mapped entirely into $U_+$ by $f$ (or into $U_-$ by $f^{-1}$). In this case, the homotopy of $h_1(S)$ obtained from the boundary compression of $Y_1$ along $D$ can be carried out transverse to $(\psi_s)$. See Figure~\ref{F:flow-compressing-disk}.  The part of the homotoped image of $Y_1$ that remains in the interior $M_f$, also determines a core for $f$ which necessarily has larger Euler characteristic.

\begin{figure}[htb!]
\def\svgwidth{3.5in}
\begingroup%
  \makeatletter%
  \providecommand\color[2][]{%
    \errmessage{(Inkscape) Color is used for the text in Inkscape, but the package 'color.sty' is not loaded}%
    \renewcommand\color[2][]{}%
  }%
  \providecommand\transparent[1]{%
    \errmessage{(Inkscape) Transparency is used (non-zero) for the text in Inkscape, but the package 'transparent.sty' is not loaded}%
    \renewcommand\transparent[1]{}%
  }%
  \providecommand\rotatebox[2]{#2}%
  \newcommand*\fsize{\dimexpr\f@size pt\relax}%
  \newcommand*\lineheight[1]{\fontsize{\fsize}{#1\fsize}\selectfont}%
  \ifx\svgwidth\undefined%
    \setlength{\unitlength}{494.08690233bp}%
    \ifx\svgscale\undefined%
      \relax%
    \else%
      \setlength{\unitlength}{\unitlength * \real{\svgscale}}%
    \fi%
  \else%
    \setlength{\unitlength}{\svgwidth}%
  \fi%
  \global\let\svgwidth\undefined%
  \global\let\svgscale\undefined%
  \makeatother%
  \begin{picture}(1,0.52774967)%
    \lineheight{1}%
    \setlength\tabcolsep{0pt}%
    \put(0,0){\includegraphics[width=\unitlength,page=1]{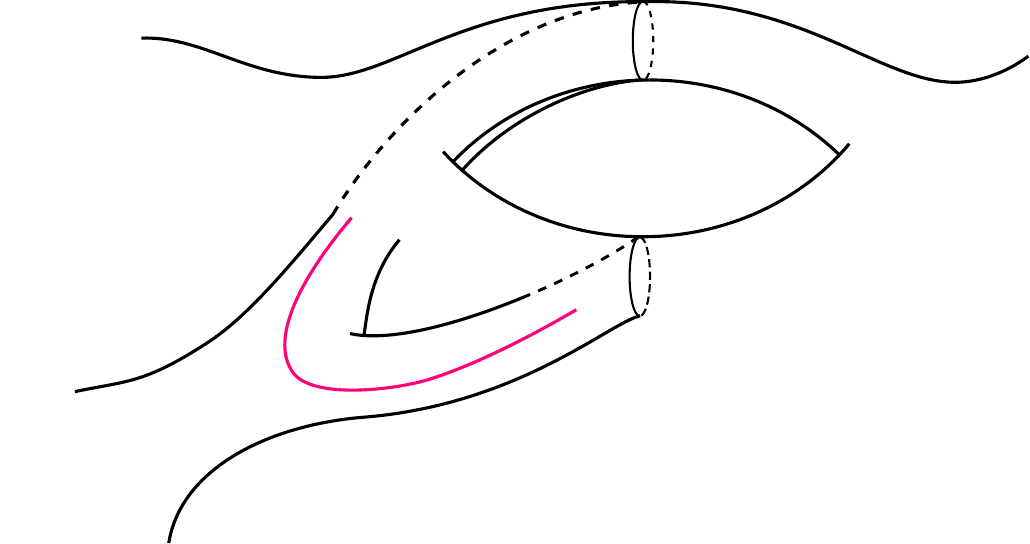}}%
    \put(0.18891449,0.10545777){\color[rgb]{0,0,0}\makebox(0,0)[rt]{\lineheight{1.25}\smash{\begin{tabular}[t]{r}$Y_1$\end{tabular}}}}%
    \put(0.98935188,0.37968739){\color[rgb]{0,0,0}\makebox(0,0)[rt]{\lineheight{1.25}\smash{\begin{tabular}[t]{r}$\partial \overline M_f$\end{tabular}}}}%
    \put(0,0){\includegraphics[width=\unitlength,page=2]{flow-compressing-disk.pdf}}%
    \put(0.26669735,0.39263959){\color[rgb]{0,0,0}\makebox(0,0)[rt]{\lineheight{1.25}\smash{\begin{tabular}[t]{r}$\beta$\end{tabular}}}}%
    \put(0.26638272,0.14442327){\color[rgb]{0,0,0}\makebox(0,0)[rt]{\lineheight{1.25}\smash{\begin{tabular}[t]{r}$\alpha$\end{tabular}}}}%
    \put(0,0){\includegraphics[width=\unitlength,page=3]{flow-compressing-disk.pdf}}%
    \put(0.33285453,0.02688302){\color[rgb]{0,0,0}\makebox(0,0)[lt]{\lineheight{1.25}\smash{\begin{tabular}[t]{l}$D$\end{tabular}}}}%
    \put(0,0){\includegraphics[width=\unitlength,page=4]{flow-compressing-disk.pdf}}%
  \end{picture}%
\endgroup%

\caption{If a boundary-compressing disk $D$ is foliated by flowlines transverse to $\alpha$ and $\beta$, as shown here, then we call it a flow-compressing disk.  Here $\alpha = h_1(\alpha')$ and $\beta = h_1(f(\alpha'))$.} \label{F:flow-compressing-disk}
\end{figure}

\begin{lemma} \label{L:reducing the core}
If $D \subset \overline M_f$ is a flow-compressing disk for the $Y$--viscera, then there is a core $Y' \subset Y$ for $f$ such that $\chi(Y') > \chi(Y)$.
\end{lemma}

\begin{proof}
Suppose that $D$ intersects $S_+$ in the arc $\beta$ (the proof in the case that $\beta \subset S_-$ is identical except with $f$ replaced by $f^{-1}$).

There is a small neighborhood $N$ of $D$ which is also a union of segments of flowlines, so that $X:= N \cap Y_1$ is a regular neighborhood of $\alpha$.
Let $Y'' = \overline{Y_1 - X}$ and set $Y' = h_1^{-1}(Y'')$, where $h_t$ is the homotopy defining $Y_1$.  Observe that $\chi(Y') > \chi(Y)$ since $Y$ is obtained from $Y'$ by adding a $1$--handle.
We can use $N$ to push $X$ along flowlines until it lies entirely inside $\partial \overline M_f$.  Concatenating this homotopy with $h_t$, we get a homotopy $g_t \colon S \to \overline M_f$ pushing along flowlines, so that $g_1^{-1}(\partial \overline M_f) = \overline{S - Y'}$.

Now observe that $Y'$ is a core for $f$.  Indeed, let $W_+$ and $W_-$ be the unions of components of $S - Y'$ containing $U_+$ and $U_-$, respectively.  Given $x \in W_+$, consider its maximal forward $(\psi_s)$--flowline in $\overline M_f$,
\[ \ell_x^+ = \bigcup_{s \geq 0} \psi_s(x),\]
where the union is over all $s \geq 0$ for which $\psi_s(x)$ is defined.  Every point of $\ell_x^+ \cap S$ is mapped by $g_1$ into $S_+$, and in particular, these points all lie in $W_+$.  On the other hand, $\ell_x^+ \cap S$ is the union of forward images of $x$ by $f$ (since the first return map of $(\psi_s)$ is $f$), and thus $f(x) \in W_+$.  Therefore, $f(W_+) \subset W_+$.
By similarly analyzing a backward flowline $\ell_x^-$, we can see that $f^{-1}(W_-) \subset W_-$, proving that $W_+$ and $W_-$ are tight nesting neighborhoods of the ends, and thus $Y'$ is a core for $f$.
\end{proof}

The next lemma says that we can promote an arbitrary $\partial$--compressing disk to a flow-compressing disk.

\begin{lemma} \label{L:good compressing disk}
If the $Y$--viscera in $\overline M_f$ is boundary compressible, then after adjusting $f$ by an isotopy which is the identity outside $Y$, there is a flow-compressing disk $D \subset \overline M_f$.
\end{lemma}
\begin{proof}
The local flow $(\psi_s)$ defines a transverse orientation to $Y_1$.  Suppose $D$ is a compressing disk with boundary arcs $\alpha \subset Y_1$ and $\beta \subset \partial \overline M_f$.  We assume that $\beta \subset S_+$, with the case $\beta \subset S_-$ proved by replacing $f$ with $f^{-1}$.  Since $D$ is a compressing disk, $D \cap Y_1 = \alpha$, and so $D$ must either be on the positive or negative side of $Y$ near $\alpha$.

If $D$ is on the negative side, we observe that there are arcs $\alpha',\beta' \subset S$ meeting in their endpoints such that $h_1(\alpha') = \alpha$ and $h_1(\beta') = \beta$. See \Cref{F:lemma-3.3}. We can piece together a map of a disk $D' \to \overline M_f$ from the homotopy $h_t$ and the disk $D$ so that the boundary of $D'$ maps homeomorphically to $\alpha' \cup \beta'$.  We lift this to a map $D' \to \widetilde M_\infty$ so that the boundary of $D'$ maps homeomorphically to $(\alpha' \cup \beta') \times \{0\}$ in $S \times \{0\}$.  Projecting $\widetilde M_\infty$ onto the first factor, this in turn defines a homotopy from $\alpha'$ to $\beta'$ in $S$, rel endpoints.  However, $\beta'$ is an arc in $\overline U_+$, while $\alpha'$ is an essential arc in $Y$.  Since an essential arc in $Y$ cannot be homotoped outside of $Y$, this is a contradiction.

\begin{figure}[htb!]
\def\svgwidth{3.5in}
\begingroup%
  \makeatletter%
  \providecommand\color[2][]{%
    \errmessage{(Inkscape) Color is used for the text in Inkscape, but the package 'color.sty' is not loaded}%
    \renewcommand\color[2][]{}%
  }%
  \providecommand\transparent[1]{%
    \errmessage{(Inkscape) Transparency is used (non-zero) for the text in Inkscape, but the package 'transparent.sty' is not loaded}%
    \renewcommand\transparent[1]{}%
  }%
  \providecommand\rotatebox[2]{#2}%
  \newcommand*\fsize{\dimexpr\f@size pt\relax}%
  \newcommand*\lineheight[1]{\fontsize{\fsize}{#1\fsize}\selectfont}%
  \ifx\svgwidth\undefined%
    \setlength{\unitlength}{554.16823586bp}%
    \ifx\svgscale\undefined%
      \relax%
    \else%
      \setlength{\unitlength}{\unitlength * \real{\svgscale}}%
    \fi%
  \else%
    \setlength{\unitlength}{\svgwidth}%
  \fi%
  \global\let\svgwidth\undefined%
  \global\let\svgscale\undefined%
  \makeatother%
  \begin{picture}(1,0.47053256)%
    \lineheight{1}%
    \setlength\tabcolsep{0pt}%
    \put(0,0){\includegraphics[width=\unitlength,page=1]{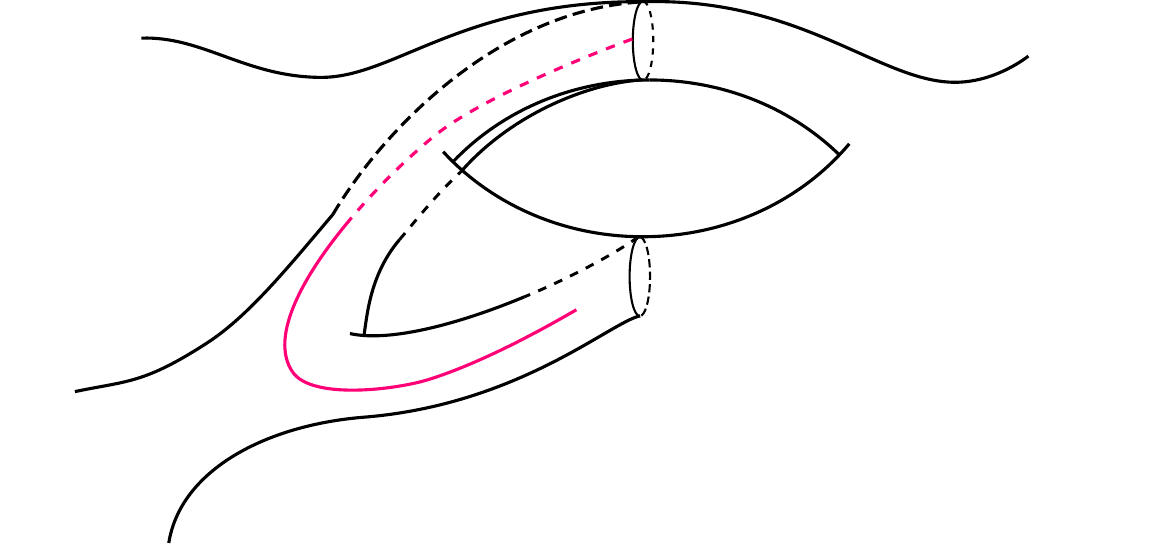}}%
    \put(0.16843292,0.09402434){\color[rgb]{0,0,0}\makebox(0,0)[rt]{\lineheight{1.25}\smash{\begin{tabular}[t]{r}$Y_1$\end{tabular}}}}%
    \put(0.1290069,0.33399525){\color[rgb]{0,0,0}\makebox(0,0)[lt]{\lineheight{1.25}\smash{\begin{tabular}[t]{l}$\partial \overline M_f$\end{tabular}}}}%
    \put(0,0){\includegraphics[width=\unitlength,page=2]{lemma-3.3.pdf}}%
    \put(0.79033189,0.34074487){\color[rgb]{0,0,0}\makebox(0,0)[lt]{\lineheight{1.25}\smash{\begin{tabular}[t]{l}$\beta$\end{tabular}}}}%
    \put(0.23750226,0.12876531){\color[rgb]{0,0,0}\makebox(0,0)[rt]{\lineheight{1.25}\smash{\begin{tabular}[t]{r}$\alpha$\end{tabular}}}}%
    \put(0,0){\includegraphics[width=\unitlength,page=3]{lemma-3.3.pdf}}%
  \end{picture}%
\endgroup%

\caption{The arcs $\alpha \subset Y_1$ and $\beta \subset \partial \overline M_f$ when $D$ is on the negative side of $Y$ near $\alpha$} \label{F:lemma-3.3}
\end{figure}

From the previous paragraph, we may assume that $D$ is on the positive side of $Y$ near $\alpha$. See \Cref{F:complicated-compressing-disk}. We can again find arcs $\alpha',\beta' \subset S$ so that $h_1(\alpha') = \alpha$ and $h_1(\beta') = \beta$, but now these arcs do not meet at their endpoints. Instead, we can find such arcs so that the endpoints of $\beta'$ are the first return points of the endpoints of $\alpha'$ by $(\psi_s)$, {\em i.e.}~the $f$--image of the endpoints of $\alpha'$.  Since flowing $\alpha'$ forward until it hits $S$, the image is precisely $f(\alpha')$, we see that $f(\alpha')$ is an arc with the same endpoints as $\beta'$.

\begin{figure}[htb!]
\def\svgwidth{7in}
\begingroup%
  \makeatletter%
  \providecommand\color[2][]{%
    \errmessage{(Inkscape) Color is used for the text in Inkscape, but the package 'color.sty' is not loaded}%
    \renewcommand\color[2][]{}%
  }%
  \providecommand\transparent[1]{%
    \errmessage{(Inkscape) Transparency is used (non-zero) for the text in Inkscape, but the package 'transparent.sty' is not loaded}%
    \renewcommand\transparent[1]{}%
  }%
  \providecommand\rotatebox[2]{#2}%
  \newcommand*\fsize{\dimexpr\f@size pt\relax}%
  \newcommand*\lineheight[1]{\fontsize{\fsize}{#1\fsize}\selectfont}%
  \ifx\svgwidth\undefined%
    \setlength{\unitlength}{889.83150402bp}%
    \ifx\svgscale\undefined%
      \relax%
    \else%
      \setlength{\unitlength}{\unitlength * \real{\svgscale}}%
    \fi%
  \else%
    \setlength{\unitlength}{\svgwidth}%
  \fi%
  \global\let\svgwidth\undefined%
  \global\let\svgscale\undefined%
  \makeatother%
  \begin{picture}(1,0.29303772)%
    \lineheight{1}%
    \setlength\tabcolsep{0pt}%
    \put(0,0){\includegraphics[width=\unitlength,page=1]{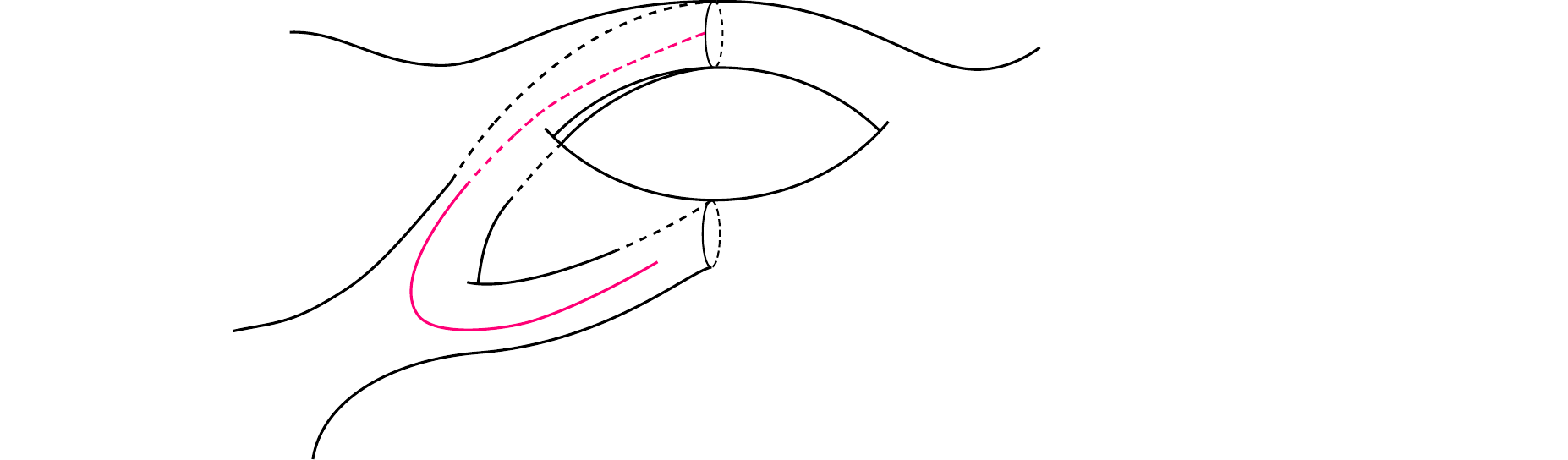}}%
    \put(0.21340727,0.05855637){\color[rgb]{0,0,0}\makebox(0,0)[rt]{\lineheight{1.25}\smash{\begin{tabular}[t]{r}$Y_1$\end{tabular}}}}%
    \put(0.60049933,0.20957634){\color[rgb]{0,0,0}\makebox(0,0)[lt]{\lineheight{1.25}\smash{\begin{tabular}[t]{l}$\partial \overline M_f$\end{tabular}}}}%
    \put(0,0){\includegraphics[width=\unitlength,page=2]{complicated-compressing-disk.pdf}}%
    \put(0.13039578,0.2246588){\color[rgb]{0,0,0}\makebox(0,0)[rt]{\lineheight{1.25}\smash{\begin{tabular}[t]{r}$\beta$\end{tabular}}}}%
    \put(0.25642219,0.08019231){\color[rgb]{0,0,0}\makebox(0,0)[rt]{\lineheight{1.25}\smash{\begin{tabular}[t]{r}$\alpha$\end{tabular}}}}%
    \put(0,0){\includegraphics[width=\unitlength,page=3]{complicated-compressing-disk.pdf}}%
    \put(0.29137972,0.01472463){\color[rgb]{0,0,0}\makebox(0,0)[lt]{\lineheight{1.25}\smash{\begin{tabular}[t]{l}$D$\end{tabular}}}}%
    \put(0,0){\includegraphics[width=\unitlength,page=4]{complicated-compressing-disk.pdf}}%
  \end{picture}%
\endgroup%

\caption{A boundary-compressing disk $D$, which is not foliated by flowlines.} \label{F:complicated-compressing-disk}
\end{figure}

Observe that if $f(\alpha') \subset U_+$, then $h_1(\alpha' \cup f(\alpha'))$ is the boundary of a flow compressing disk, and we are done.  Because we have not imposed any constraints on the behavior of $f$ inside of $Y$, we may not have $f(\alpha') \subset U_+$, in which case we claim we can adjust $f$ by an isotopy supported in $Y$ so that the new homeomorphism does send $\alpha'$ into $U_+$.

To find the required isotopy, we first observe that using the disk $D$, we can construct a homotopy, rel endpoints, from $f(\alpha')$ to $\beta'$.  To do this, we consider the ``rectangle" which is a union of the flowlines between $\alpha'$ and $f(\alpha')$ in $\overline M_f$, and drag this along via the homotopy $h_t$ to define a disk $D' \subset \overline M_f$ whose boundary is the union of the two arcs $h_1(\alpha') = \alpha$ and $h_1(f(\alpha'))$, and whose interior is contained in a component of the complement of $\partial \overline M_f \cup Y_1$.  Since both disks $D$ and $D'$ have their interiors in the same complementary component of $\overline M_f \cup Y_1$, their union $D \cup D'$ can be pushed forward via the flow into the branched surface $h_1(S)$, and then lifted back to $S$ to define the homotopy, rel endpoints from $f(\alpha')$ to $\beta'$.

Now we assume, as we may, that $f(\partial Y)$ meets $\partial Y$ transversely and minimally.  Then since $\beta' \subset U_+$ is an arc disjoint from $\partial Y$ and it is homotopic, rel endpoints, to $f(\alpha')$ which meets $f(\partial Y)$ only in its endpoints, it follows that we may postcompose $f$ with an isotopy that is the identity outside $f(Y)$, so that $f(\alpha')$ is disjoint from $\partial Y$, and thus contained in $U_+$.  This is equivalent to precomposing $f$ by an isotopy that is the identity outside $Y$.  Replacing $f$ with this isotopic homeomorphism, $f(\alpha') \subset U_+$, and thus $\alpha = h_1(\alpha')$ and $h_1(f(\alpha'))$ defines a flow-compressing disk, as required.
\end{proof}
From the two lemmas above, we deduce the following.
\begin{proposition} \label{P:minimal core bdy-inc}
If $Y \subset S$ is a minimal core for a strongly irreducible end-periodic homeomorphism and $Y_1 \subset \overline M_f$ is the $Y$--viscera, then $Y$ is boundary incompressible.
\end{proposition}
\begin{proof}
If $Y_1$ were boundary compressible, \Cref{L:good compressing disk} would providuce a flow-compressing disk, then \Cref{L:reducing the core} would produce a new core $Y' \subset Y$ with $\chi(Y') > \chi(Y)$, which contradicts minimality of $Y$. 
\end{proof}

\section{Pleated surfaces}
\label{section.pleated}
For the remainder of this section, we assume that $f \colon S \to S$ is a strongly irreducible end-periodic homeomorphism.  Throughout, $Y \subset S$ will denote a core for~$f$.

A {\em pants--lamination} on a hyperbolic surface $X$ is a geodesic lamination whose leaves are the curves of a pants decomposition together with isolated leaves such that each complementary component is an ideal triangle spiraling into all three cuffs of its pair of pants, see \Cref{fig:ideal-triangulation-pants}.  We also make a technical assumption that for each pants curve, the leaves that spiral in towards that curve do so ``in the same direction" on both sides of the curve. This utility in this assumption is clarified in \Cref{S:interpolation} (see the remark at the end of that section for an alternative approach). See \Cref{fig:pants-lamination-direction} for an illustration of this behavior.

\begin{figure}[htb!]
    \centering
    \includegraphics[width = 2in]{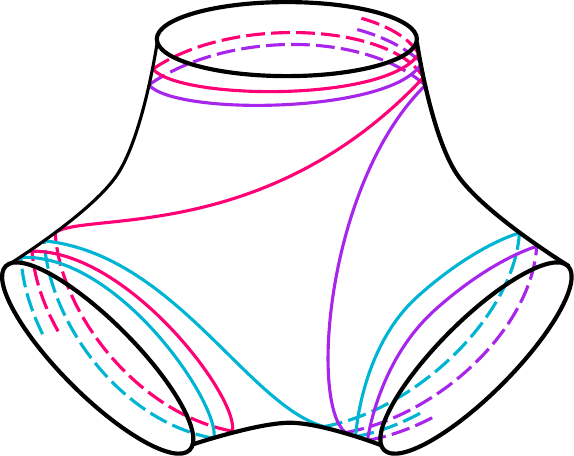}
    \caption{An ideal triangulation of a pair of pants. The three biinfinite geodesics shown are examples of the isolated leaves found in a pants--lamination.
    }    
    \label{fig:ideal-triangulation-pants}
\end{figure}

\begin{figure}[htb!]
    \centering
    \includegraphics[width = 2in]{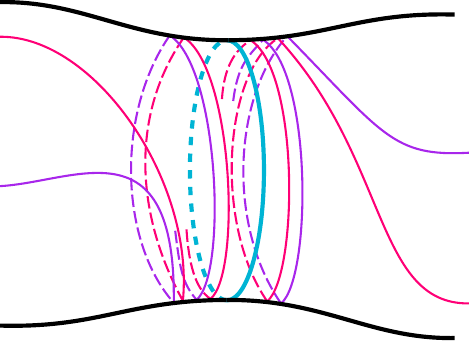}
    \caption{The blue pants curve in the middle has leaves of the pants-lamination spiraling in towards it ``to the right" from both sides or ``to the left" from both sides.}    
    \label{fig:pants-lamination-direction}
\end{figure}

Let $M$ be a hyperbolic 3-manifold.  A \textit{pleated surface} is an arc-length preserving map $\varphi: (X, \sigma) \to M$ from a surface $X$ with a complete hyperbolic metric $\sigma$ such that: (a) $\varphi$ maps leaves of some $\sigma$--geodesic lamination $\lambda \subset X$ to geodesics, and (b) $\varphi$ is totally geodesic in the complement of $\lambda$. We say that $\lambda$ is the \textit{pleating locus} of $\varphi$ if it is the smallest lamination satisfying conditions (a) and (b). See \cite[Section 5]{CanaryEpsteinGreen} for more details.

We introduce the following definition as an adaption of pleated surfaces to our infinite-type setting. A {\em well-pleated surface adapted to the core $Y$} in $\overline M_f$ is a pleated surface $\varphi \colon (S,\sigma) \to \overline M_f$, homotopic to the inclusion of $S$ into $\overline M_f$, such that
\begin{enumerate}
    \item \label{It:pants} the pleating locus is contained in a pants--lamination, and
    \item \label{It:core} $\varphi(S-Y) \subset \partial \overline M_f$.
\end{enumerate}
Observe that ``the inclusion" of $S$ into $\overline M_f$ is really only well-defined up to precomposing with powers of $f$.  We will later avoid this issue by passing to the cover $\widetilde M_\infty \subset S \times [-\infty,\infty]$, where this ambiguity disappears.  We also note that the metric $\sigma$ on $S$ is compatible with $Y$ and induced by the metric on $\widetilde M_\infty$.

One may follow \cite[Theorem~5.3.6]{CanaryEpsteinGreen} to construct a well-pleated surface representing the inclusion, adapted to any given core $Y$.  For this, construct a homotopy $h_t \colon S \to \overline M_f$ as in \Cref{S:viscera defined}.  We replace $h_1$ with a homotopic map $\varphi' \colon S \to \overline M_f$ so that $\varphi'$ sends each component of $\partial Y$ to a closed geodesic in $\partial \overline M_f$.  Since $h_1$ was a local homeomorphism from the neighborhoods of the ends $S-Y$ onto $\partial \overline M_f$, we can assume that $\varphi'$ is as well. We choose a hyperbolic structure on $S-Y$ so that $\varphi'|_{S-Y}$ is a local isometry.  Now we choose any pants lamination $\lambda$ containing $\partial Y$, and observe that we are now reduced to constructing a pleated surface $Y \to \overline M_f$ homotopic to $\varphi'|_{\partial Y}$, rel $\partial Y$, realizing the finite lamination $\lambda \cap Y$.  We can view this restriction $\varphi'|_Y \colon (Y,\partial Y) \to (\overline M_f,\partial \overline M_f)$ as a pleated surface representative of the $Y$--viscera, $(Y_1,\partial Y_1) \subset (\overline M_f,\partial \overline M_f)$, realizing the finite lamination $h_1(\lambda \cap Y) \subset Y_1$.

We let 
\[ 
    \Omega(f,Y) = \{ \varphi \colon (S,\sigma) \to \overline M_f \mid \varphi \mbox{ is well-pleated adapted to } Y \}
\] 
denote the set of all well-pleated surfaces adapted to $Y$.  A well pleated surface adapted to {\em some} core will simply be called a well-pleated surface and we denote the set of all such by $\Omega(f)$.

In the proof of the following lemma, we will make use of the theorem of Basmajian \cite[Theorem~1.1]{Basmajian} that there is a constant $\beta > 0$, depending only on $\xi(f)$, so that the $\beta$--neighborhood of $\partial \overline M_f$ is a product, $\partial \overline M_f \times [0,\beta]$. The proof employs a standard area argument. See, for example, Canary's proof of the Bounded Diameter Lemma \cite[Lemma~4.5]{Canary.CoveringTheorem}.

\begin{lemma} [Bounding the boundary] \label{L:bounding the boundary}
Suppose that $Y$ is a minimal core of the strongly irreducible, end-periodic homeomorphism $f \colon S \to S$ and suppose \\ ${(\varphi \colon (S,\sigma) \to \overline M_f) \in \Omega(f,Y)}$, then the total $\sigma$--length of the boundary of $\partial Y$ satisfies $\ell_\sigma(\partial Y) \leq \frac{2\pi |\chi(Y)|}{\beta}$.
\end{lemma}
\begin{proof}
Take $\epsilon >0$ maximal such that the open $\epsilon$--neighborhood (with respect to the metric $\sigma$) of $\partial Y$ are annuli. 
The boundary of this neighborhood meets itself at some point, producing an essential arc of length $2\epsilon$ in $Y$.  
The area of the $\epsilon$--neighborhood is no more than the total area of $(Y,\sigma)$, which is $-2 \pi \chi(Y) = 2 \pi |\chi(Y)|$ by the Gauss--Bonnet formula. 
On the other hand, the area of this $\epsilon$--neighborhood is $ \sinh(\epsilon) \ell_\sigma(\partial Y) > \epsilon \ell_\sigma(\partial Y)$, and so
\[ \epsilon \ell_\sigma(\partial Y)  < 2 \pi |\chi(Y)|.  
\]

If $\ell_\sigma(\partial Y) > \frac{2\pi |\chi(Y)|}{\beta}$, then $\epsilon < \beta$.  In this case, the closure of the $\epsilon$--neighborhood of $\partial Y$ would map into the $\beta$--neighborhood of $\partial \overline M_f$.  It follows that $\alpha$ is mapped into the collar neighborhood of $\partial \overline M_f$, and is hence properly homotopic into $\partial \overline M_f$.  But $\varphi|_Y$ is a pleated surface representative of $Y_1$, and thus \cite[Theorem~2.2]{AitchisonRubinstein1993} implies that $Y_1$ is boundary-compressible.  This contradicts \Cref{P:minimal core bdy-inc}, since $Y$ was a minimal core.  Therefore, $\ell_\sigma(\partial Y) \leq \frac{2\pi |\chi(Y)|}{\beta}$, as required. 
\end{proof}

A {\em minimally well-pleated surface} is any well-pleated surface $\varphi \in \Omega(f,Y)$ for which $Y$ is a minimal core and where the pleating locus contains $\partial Y$.

\begin{corollary} \label{C:boundary length bound}
If $\varphi \in \Omega(f)$ is minimally well-pleated and $\beta(\xi(f))>0$ is Basmajian's constant, then
\[ \ell_\sigma(\partial \Sigma_\varphi) \leq \frac{2\pi |\chi(f)|}{\beta(\xi(f))}.\]
\end{corollary}

\section{Interpolation and pants}
\label{section.interpolation}
In this section, we describe how to produce paths of pants decompositions that we will use to prove the required lower bounds on volume.  In what follows, we typically assume $Y$ is a minimal core for a strongly irreducible end-periodic $f$, so that $\chi(Y) = \chi(f)$. 

\subsection{Pants and cores} Suppose $(\varphi \colon (S,\sigma) \to \overline M_f) \in \Omega(f,Y)$ is any minimally well-pleated surface (\textit{i.e.}~$Y$ is a minimal core), and let $U_+ \cup U_- = S-Y$ be the associated tight nesting neighborhoods.  Applying an isotopy if necessary, we assume that for any component $\alpha$ of $\partial_- Y$, $f^{-1}(\alpha) $ is either an essential curve of $U_-$ or else a different component of $\partial_-Y$. For example, this can be arranged without affecting previous assumptions if the isotopy class of $Y$ is realized with geodesic boundary.

We also set
\[ \Delta_+ = \overline{U_+ - f(U_+)},\]
as in \Cref{S:complexities}, which is a fundamental domain for the action of $f$ on $\mathcal U_+$. Note that $\Delta_+$ is the subsurface bounded by (some components of) $\partial_+ Y$ and $f(\partial_+ Y)$, by \Cref{L:intution matched}, since $Y$ is balanced by minimality.

Now we define
\[ 
    \Sigma  = Y  \cup \Delta_+,
\]
(see \Cref{fig:defining-sigma-phi-k}) and observe that from the discussion in \Cref{S:complexities} we have
\begin{equation} \label{Eq:ComplexityBound}
    \xi(\Delta_+) \leq \tfrac32 |\chi(\Delta_+)| = \tfrac34|\chi(\partial \overline M_{f})| = \tfrac{1}{2} \xi(f).
\end{equation}

\begin{figure}[htb!]
    \def\svgwidth{3.5in}
\begingroup%
  \makeatletter%
  \providecommand\color[2][]{%
    \errmessage{(Inkscape) Color is used for the text in Inkscape, but the package 'color.sty' is not loaded}%
    \renewcommand\color[2][]{}%
  }%
  \providecommand\transparent[1]{%
    \errmessage{(Inkscape) Transparency is used (non-zero) for the text in Inkscape, but the package 'transparent.sty' is not loaded}%
    \renewcommand\transparent[1]{}%
  }%
  \providecommand\rotatebox[2]{#2}%
  \newcommand*\fsize{\dimexpr\f@size pt\relax}%
  \newcommand*\lineheight[1]{\fontsize{\fsize}{#1\fsize}\selectfont}%
  \ifx\svgwidth\undefined%
    \setlength{\unitlength}{898.43189526bp}%
    \ifx\svgscale\undefined%
      \relax%
    \else%
      \setlength{\unitlength}{\unitlength * \real{\svgscale}}%
    \fi%
  \else%
    \setlength{\unitlength}{\svgwidth}%
  \fi%
  \global\let\svgwidth\undefined%
  \global\let\svgscale\undefined%
  \makeatother%
  \begin{picture}(1,0.48162045)%
    \lineheight{1}%
    \setlength\tabcolsep{0pt}%
    \put(0,0){\includegraphics[width=\unitlength,page=1]{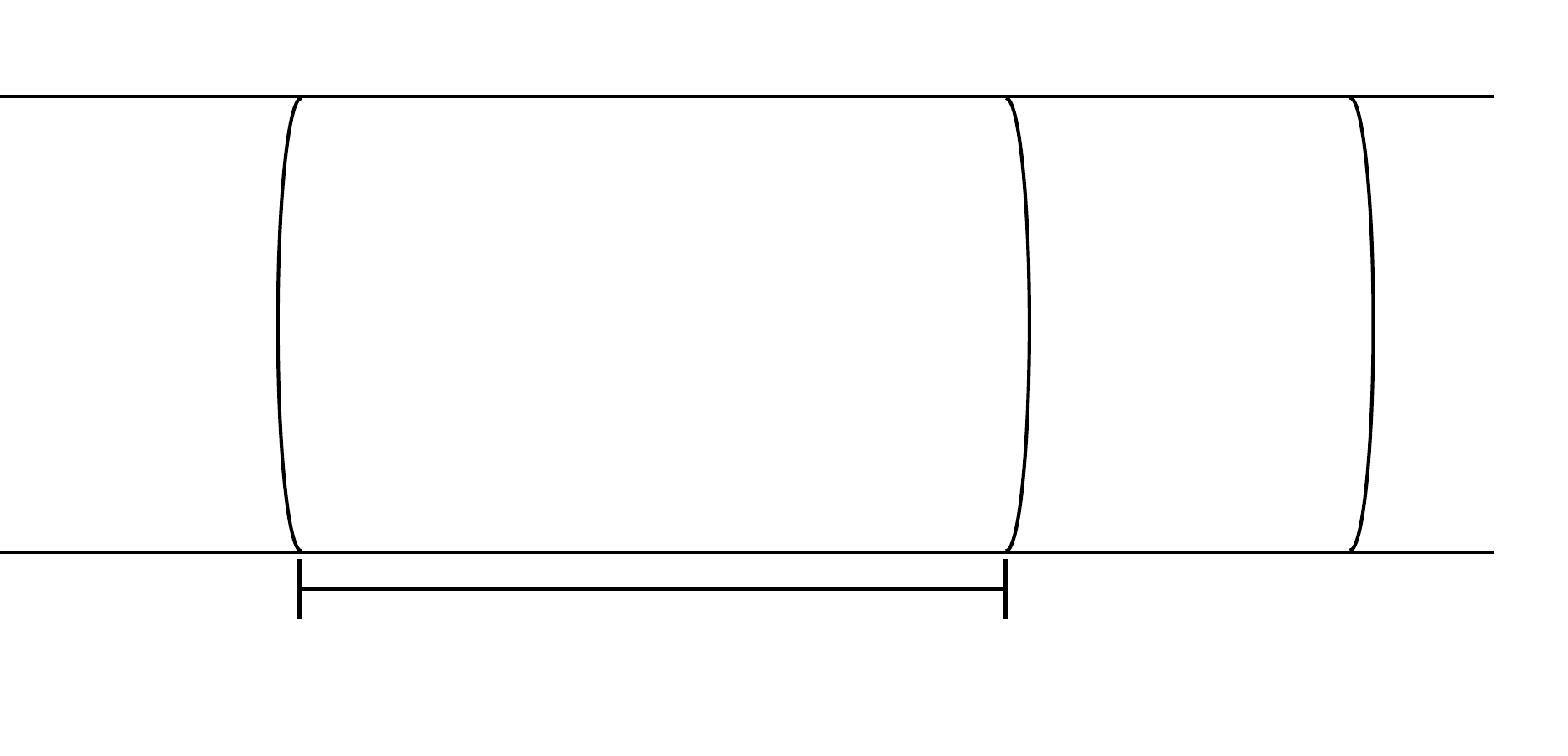}}%
    \put(0.41575877,0.06280457){\color[rgb]{0,0,0}\makebox(0,0)[t]{\lineheight{1.25}\smash{\begin{tabular}[t]{c}$Y$\end{tabular}}}}%
    \put(0,0){\includegraphics[width=\unitlength,page=2]{defining-sigma-phi-k.pdf}}%
    \put(0.75109537,0.45592472){\color[rgb]{0,0,0}\makebox(0,0)[t]{\lineheight{1.25}\smash{\begin{tabular}[t]{c}$\Delta_+$\end{tabular}}}}%
    \put(0,0){\includegraphics[width=\unitlength,page=3]{defining-sigma-phi-k.pdf}}%
    \put(0.52882339,0.00689676){\color[rgb]{0,0,0}\makebox(0,0)[t]{\lineheight{1.25}\smash{\begin{tabular}[t]{c}$\Sigma$\end{tabular}}}}%
    \put(-0.00252718,0.26493239){\color[rgb]{0,0,0}\makebox(0,0)[lt]{\lineheight{1.25}\smash{\begin{tabular}[t]{l}$U_-$\end{tabular}}}}%
    \put(1.00234783,0.26493239){\color[rgb]{0,0,0}\makebox(0,0)[rt]{\lineheight{1.25}\smash{\begin{tabular}[t]{r}$U_+$\end{tabular}}}}%
  \end{picture}%
\endgroup%

    \caption{A schematic of the surface $S$ with $Y$, $\Delta_+$, and $\Sigma$ labelled. Here $\Delta_+$ is a fundamental domain for the action of $f$ on $\mathcal U_+$, and $\Sigma$ is the union of the minimal core, $Y$, with $\Delta_+$.} 
    \label{fig:defining-sigma-phi-k}
\end{figure}

\noindent We are now ready to construct a bounded-length, $f$-asymptotic pants decomposition of $S$.

\begin{lemma} \label{L:bounded pants decomp} There exists a constant $L = L(\chi(f),\xi(f))$ so that for any minimally well-pleated surface ${(\varphi \colon (S,\sigma) \to \overline M_f) \in \Omega(f, Y)}$, there exists a pants decomposition $P$ of $S$ such that
\begin{enumerate}
    \item \label{It:extending boundary} $\partial \Sigma \subset P$;
    \item \label{It:move support} $P$ and $f(P)$ differ only on $\Sigma$; and
    \item \label{It:bounded length} each curve in $P$ has $\sigma$--length at most $L$.
\end{enumerate}

Furthermore, since $P$ contains $\partial Y$, there exists a well-pleated surface that realizes a pants lamination containing $P$.
\end{lemma}
\begin{proof} Fix $(\varphi \colon (S,\sigma) \to \overline M_f) \in \Omega(f,Y)$ minimally well-pleated and adapted to a core $Y$, and continue with the notation as above, so that $U_\pm$ are the tight nesting neighborhoods defined by $Y$.
Since $\varphi(U_\pm)$ is entirely contained in $\partial \overline M_f$, it follows that, with respect to the metric $\sigma$, after an isotopy we can assume that $f$ isometrically maps $U_+$ into itself and $f^{-1}$ isometrically maps $U_-$ into itself (ensuring that $\sigma$ and $f$ satisfy prior assumptions).

We now construct the desired pants decomposition $P$ of $S$. We start with the curves $\partial Y$. Note that $|\chi(Y)|= |\chi(f)|$ and
\[ \ell_\sigma(\partial Y) \leq \frac{2\pi |\chi(Y)|}{\beta(\xi(f))} \leq \frac{2\pi |\chi(f)|}{\beta(\xi(f))},\]
by \Cref{C:boundary length bound}.
Since $\xi(Y) \leq \frac{3}2 |\chi(Y)|$,  \Cref{thm:relative-bers} then guarantees that we can choose a pants decomposition of $Y$ so that each pants curve has length bounded by some $L_0 = L_0(\chi(f),\xi(f))$. We require $P$ to contain this pants decomposition of $Y$ (including the boundary curves $\partial Y$).

From \Cref{Eq:ComplexityBound}, $\xi(\Delta_-) \leq \frac12\xi(f)$. Since $\ell_\sigma(\partial Y)$ is uniformly bounded and since $f$ isometrically maps $U_+$ into itself, we see that there is also a uniform bound on the length of $\partial \Delta_+$ in terms of $\chi(f)$ and $\xi(f)$.
So again applying \Cref{thm:relative-bers}, there is a pants decomposition $P_+$ of $\Delta_+$ (including the boundary curves, $\partial \Delta_+$) such that each pants curve has length bounded by some $L_1 = L_1(\chi(f), \xi(f))$.  Then, since
\[ 
    U_+ = \bigcup_{j=0}^\infty f^{j} (\Delta_+) 
\]
with any two distinct translates $f^{j}\Delta_+$ and $f^{j'}\Delta_+$ intersecting at most in their boundary curves, we can extend $P$ over $U_+$ so that
\[ 
    U_+ \cap P = \bigcup_{j=0}^\infty f^{j} (P_+),
\]
We similarly construct $P$ on $U_-$ with each curve's length bounded by some constant $L_2 = L_2(\chi(f), \xi(f))$.

Now set $L = \max\{L_0,L_1,L_2\}$.  The lemma follows by construction: \Cref{It:bounded length} is by definition, and \Cref{It:extending boundary} and \Cref{It:move support} are a consequence of the fact that $\Sigma = Y\cup \Delta_+$. 
\end{proof}

For the remainder of this section, fix $(\varphi \colon (S,\sigma) \to \overline M_f) \in \Omega(f, Y)$ with $\chi(Y) = \chi(f)$ and $L = L(\chi(f), \xi(f))$, and $P$ as in \Cref{L:bounded pants decomp}. Next, fix the lift $\widetilde \varphi \colon S \to \widetilde M_\infty \subset S \times [-\infty,\infty]$ of $\varphi$ so that $\widetilde \varphi$ is homotopic to the identity after projecting onto the first factor.

Composing with the covering transformation $F$ on $\widetilde M_{\infty}$, we get $F \circ \widetilde \varphi \colon S \to \widetilde M_\infty$, which is another lift of $\varphi$, but rather than being homotopic to the identity on $S$ (after projecting onto the first factor) it is homotopic to $f$, since $F(x,t) = (f(x),t-1)$.  Therefore,
\[ F \circ \widetilde \varphi \circ f^{-1} \colon S \to \widetilde M_\infty\] 
is homotopic to the identity (again, after projecting to the first factor).

Pulling back the metric by $f^{-1}$ gives a new (lift of a) well-pleated surface
\[ 
    F \circ \widetilde \varphi \circ f^{-1} \colon (S,(f^{-1})^*\sigma) \to \widetilde M_\infty.
\]
Since $P$ has length at most $L$ with respect to $\sigma$, $f(P)$ has length at most $L$ with respect to $(f^{-1})^*\sigma$.

By \Cref{L:bounded pants decomp}, $P$ and $f(P)$ differ only in $\Sigma$, and we define
    \[ P_{\alpha} = P \cap \Sigma \quad \mbox{ and } \quad P_{\omega} = f(P) \cap \Sigma.\]
These are pants decompositions of $\Sigma$ such that
\begin{equation} \label{E:reduction to finite type pants}
    d_{\mathcal P(S)}(P,f(P)) \leq d_{\mathcal P(\Sigma)}(P_{\alpha},P_{\omega}),
\end{equation} 
where
\begin{equation} \label{Eq:Big Sigma Euler}
    |\chi(\Sigma)| =   |\chi(\Delta_-)| + |\chi(Y)| = \tfrac13 \xi(f) + |\chi(f)| < \xi(f) + |\chi(f)| ,
\end{equation} 
and each curve of $\partial \Sigma$ has length at most $L$.

To simplify the notation, we write $\phi_{\alpha} \colon \Sigma \to \widetilde M_\infty$ to denote the restriction of $\widetilde \varphi$ to $\Sigma$ and $\phi_{\omega} \colon \Sigma \to \widetilde M_\infty$ to denote the restriction of $F \circ \widetilde \varphi \circ f^{-1}$ to $\Sigma$. As the notation suggests, $\phi_{\alpha}$ maps the curves of $P_{\alpha}$ to geodesics of length at most $L$ and $\phi_{\omega}$ maps the curves of $P_{\omega}$ to geodesics of length at most $L$.  We write $\sigma_{\alpha}$ and $\sigma_{\omega}$ to denote the hyperbolic structures so that $\phi_{\alpha} \colon (\Sigma,\sigma_{\alpha}) \to \widetilde M_\infty$ and $\phi_{\omega} \colon (\Sigma,\sigma_{\omega}) \to \widetilde M_\infty$ are pleated surfaces (representing the $\Sigma$--viscera).  Since both $\phi_{\alpha}$ and $\phi_{\omega}$ map $\partial \Sigma$ to geodesics, by precomposing one of these with an isotopy of the identity, we assume (as we may) $\phi_{\alpha}$ and $\phi_{\omega}$ agree on the boundary and that they are homotopic by a homotopy that is stationary on $\partial \Sigma$.

\subsection{Simplicial hyperbolic surfaces}

\label{sub:simp-hyp-surfaces}

To produce continuous families of ``good" representatives of a homotopy class of $\phi_{\alpha} \colon \Sigma \to \widetilde M_\infty$ (and hence of $\phi_{\omega}$) we use simplicial hyperbolic surfaces, following \cite{Canary.CoveringTheorem}. 

We fix once and for all $k_0 = |\partial \Sigma|$ points, one on each component of $\partial \Sigma$. Let $k \geq k_0$. A \textit{$k$--simplicial pre-hyperbolic surface in $\widetilde M_\infty$} is a map $\eta \colon \Sigma \to \widetilde M_\infty$ that satisfies the following:
\begin{itemize}
    \item There is a triangulation\footnote{This is not a triangulation in the classical sense, but rather a $\Delta$--complex structure in the sense of \cite{Hatcher.2002}} ${\mathcal T}$ of $\Sigma$ with $k$ vertices, exactly one on each boundary component at the fixed $k_0$ points, such that $\eta$ takes each triangle to a non-degenerate totally geodesic triangle in $\widetilde M_\infty$.
    \item The restriction of $\eta$ to each component of $\partial \Sigma$ is a closed geodesic.
    \item The map $\eta$ is homotopic to $\phi_{\alpha}$ through maps that are stationary on $\partial \Sigma$.
\end{itemize}
Note that, for such a surface, each component of $\partial \Sigma$ is parameterized by a single edge of the triangulation under $\eta$.
Furthermore, the hyperbolic metrics on the triangles induce a singular hyperbolic metric $\sigma_\eta$ on $\Sigma$. In this metric, the boundary is a smooth geodesic, except possibly at the vertices.
The $k$--simplicial pre-hyperbolic surface is a $k$--simplicial \textbf{\textit{hyperbolic}} surface if the cone angles in the interior are all at least $2\pi$, and those on the boundary are at least $\pi$.
In particular note that a $k_0$--simplicial pre-hyperbolic surface is automatically hyperbolic.
The set of all $k$--simplicial hyperbolic surfaces is denoted $\SH_k$, and the set of all simplicial hyperbolic surfaces by $\SH = \cup_k \SH_k$, and we equip both with the compact--open topology.
We will be primarily interested in $k$--simplicial hyperbolic surfaces when $k = k_0$ or $k=k_0+1$.

The universal cover $\widetilde M$ of $\widetilde M_\infty$ may be identified with a convex subset of $\mathbb H^3$ whose frontier is a union of totally geodesic hyperbolic planes. This allows the identification of $\overline{M}_f$ as the quotient of such a subset. We choose such an identification, as well as an equivariant lift $\widetilde \Sigma \to \widetilde M$ of $\phi_{\alpha}$, which also gives us an equivariant lift $\widetilde \eta$ of any $k$--simplicial hyperbolic surface $\eta$ by lifting the homotopy to $\phi_{\alpha}$. 

We let $\SH({\mathcal T})\subset \SH$ be the subspace consisting of all simplicial hyperbolic surfaces whose underlying triangulation is ${\mathcal T}$.

\begin{lemma} \label{L:realize simplicial}
For any triangulation ${\mathcal T}$ of $\Sigma$ with $k_0$ vertices, one vertex at each of the fixed points on $\partial \Sigma$, the space $\SH({\mathcal T})$ is non-empty and any two elements of $\SH({\mathcal T})$ differ by precomposing with a homeomorphism isotopic to the identity by an isotopy that is stationary on the boundary.
\end{lemma}

\begin{proof} We will first show that $\SH(\mathcal{T})$ is nonempty by constructing an element $\eta: \Sigma \to \widetilde M_{\infty}$ of $\SH({\mathcal T})$ inductively over the skeleta.

We begin by declaring that $\eta$ agrees with $\phi_{\alpha}$ on the boundary of $\Sigma$, and hence the vertices of ${\mathcal T}$, which is required for any element of $\SH(\mathcal T)$.
Let $e$ be an edge of ${\mathcal T}$.
If the endpoints of $e$ are distinct, then, by lifting to $\mathbb H^3$ and using a straight--line homotopy, we may homotope $\phi_{\alpha}\big|_e$ relative to its endpoints to a geodesic segment.
If this is the case, we define $\eta$ to send $e$ to this geodesic segment which necessarily has positive length: the endpoints are either on distinct boundary components of $\partial \widetilde M_\infty$, or on disjoint closed geodesics in a single component of $\partial \widetilde M_\infty$. 
If $e$ is a loop, then it is homotopically nontrivial, and hence it is carried by $\phi_{\alpha}$ to an essential loop, and we homotope $\phi_\alpha$ by a straight--line homotopy again to $\eta$ mapping $e$ to the geodesic representative of the based homotopy class of loops (if the loop is a boundary component, it is already geodesic and $\eta$ and $\phi_\alpha$ agree there).
The boundary $\partial T$ of any triangle $T$ in ${\mathcal T}$ is null--homotopic in $\Sigma$, and so it also lifts to the universal cover where we can extend our straight line homotopy to a homotopy of $\phi_{\alpha}|_T$ to $\eta|_T$ mapping $T$ to a (possibly degenerate) geodesic triangle immersed in $\widetilde M_\infty$.  Having done this for every triangle $T$, we have defined $\eta: \Sigma \to \widetilde M_{\infty}$ and the homotopy from $\phi_\alpha$.

We claim that every geodesic triangle $\eta|_T$ is in fact non-degenerate.
This means that the lift $\widetilde \eta|_T$ to $\mathbb H^3$ is an embedding of a geodesic triangle.  
Suppose that this is not the case.  First, observe that every edge of $T$ in $\Sigma$ either connects distinct vertices or is a non-null homotopic loop, so the image of each edge is a non-degenerate segment.  Thus, the only degeneracy that may occur in $\widetilde \eta|_T$ is that all three vertices lie on a single geodesic segment.  All three of these points lie on the boundary of $\widetilde M$, which we recall is a convex subset of $\mathbb H^3$ bounded by hyperbolic planes.  Therefore, the entire segment must lie in the boundary.

Now, projecting the segment back to $\widetilde M_\infty$, we obtain a segment in $\partial \widetilde M_\infty$ passing through three vertices in $\eta(\partial \Sigma)$.  Note that the segment cannot be entirely contained in $\eta(\partial \Sigma)$ since then all three edges of $T$ are mapped to the homotopy class of the boundary loop, which is not allowed in the triangulation.  Therefore, the segment defines a geodesic path in a component of $\partial \widetilde M_\infty$, and hence in either $\mathcal U_+ \times \{\infty\}$ or $\mathcal U_- \times \{-\infty\}$.  Moreover, this segment meets the geodesic boundary $\eta(\partial \Sigma)$ transversely.  A subsegment of the path between two of the vertices enters one of $U_+ \times \{\infty\}$  or $U_- \times \{-\infty\}$ from the vertex.  Projecting $\widetilde M_\infty$ onto the first factor $S$, this subsegment projects to an essential path in $S$ which cannot be homotoped entirely contained in $\Sigma$.  This is impossible because the path is homotopic to an edge of the triangulation, and hence an essential arc in $\Sigma$.  Therefore, the triangle $T$ is non-degenerate.

The link of a vertex $v$ in ${\mathcal T}$ defines a path in the sphere $T^1_{\widetilde \eta(v)}(\mathbb H^3)$ joining antipodal points---namely, the two tangent vectors to the boundary geodesic---and thus has length at least $\pi$, proving that the cone angle at the vertices is at least $\pi$ (\textit{c.f.}~the NLSC [not locally strictly convex] property and \cite[Lemma~4.2]{Canary.CoveringTheorem}).  Therefore, after reparameterizing if necessary, we conclude that $\eta$ is the desired $k_0$--simplicial hyperbolic surface and so, $\SH(\mathcal T)$ is non-empty.

Now, given any $\eta,\eta' \in \SH(\mathcal T)$, the two maps are homotopic to $\phi_\alpha$ by a homotopy that is stationary on $\partial \Sigma$, they are also homotopic to each other by such a homotopy.  Since each edge $e$ of $\mathcal T$ must be sent to the geodesic in the relative homotopy class of $\phi_\alpha\big|_e$, it follows that $\eta$ and $\eta'$ differ on each edge by a reparameterization constant on the endpoints. Therefore, $\eta$ and $\eta'$ differ by reparameterization that maps simplicies to simplices.  Since each triangle maps to a non-degenerate triangle, the reparameterization is necessarily a homeomorphism preserving the triangulation which is the identity on the vertices.  It follows that this reparameterizing homeomorphism is isotopic to the identity rel the vertices, and thus $\eta$ and $\eta'$ differ by precomposing by a homeomorphism isotopic to the identity rel the vertices and boundary.
\end{proof}

For each component $\gamma \subset \partial \Sigma$, $\phi_\alpha(\gamma) = \phi_\omega(\gamma)$ is a closed geodesic in the totally geodesic boundary $\partial \widetilde M_\infty$ of length at most $L$. Consider the annular cover of the component of $\partial \widetilde M_\infty$ to which this curve lifts.  This annulus is divided by (the lift of the image of) $\gamma$ into two half-open annuli with boundary $\gamma$, and we let $\breve{\Sigma} \supset \Sigma$ be obtained by gluing one of these half-open annuli to each boundary component of $\Sigma$.  For any $\eta \colon (\Sigma,{\mathcal T}) \to \widetilde M_\infty \in \SH$, we have $\eta|_\gamma = \phi_\alpha|_\gamma$, and so we can extend $\eta$ to a map
\[ \breve{\eta} \colon \breve{\Sigma} \to \widetilde M_\infty,\]
whose restriction to the added half-open annuli is the restriction of the covering map to $\partial \widetilde M_\infty$.  The singular hyperbolic metric $\sigma_\eta$ extends to a singular hyperbolic metric of the same name on $\breve{\Sigma}$, so that $\breve{\eta}$ is a local isometry on each half-open annulus.
Then let $\sigma_\eta^{\hyp}$ be the hyperbolic uniformization of the conformal structure on $(\breve{\Sigma},\sigma_\eta)$. We can do this since taking a small open ball around a cone point and deleting the cone point results in something conformally equivalent to the unit disk minus a point. We write $\ell_{\sigma_\eta^\hyp}(\gamma)$ for the length of the $\sigma_\eta^\hyp$--geodesic representative of any essential closed curve $\gamma$ in $\breve{\Sigma}$.

\begin{theorem}\label{T:Ahlfors}
For any $\eta \colon (\Sigma,{\mathcal T}) \to \widetilde M_\infty$ in $\SH$, the identity map
\[ 
    \id \colon (\breve{\Sigma},\sigma_\eta^\hyp) \to (\breve{\Sigma},\sigma_\eta)
\]
is $1$--Lipschitz. Furthermore, for each component $\gamma \subset \partial \Sigma$, the $\ell_{\sigma_\eta^\hyp}(\gamma) \leq 2L$.
\end{theorem} 
\begin{proof}
The first statement is a consequence of the Ahlfors--Schwartz--Pick Theorem \cite{Ahlfors.1938}.  For the second statement, we will use a modulus argument (see \cite[Theorem~2.16.1]{epstein-marden.2006} for a precise statement of the correspondence between modulus of an annulus and the length of its core curve). Note that the length of $\phi_\alpha(\gamma) = \eta(\gamma)$ in $\widetilde M_\infty$ is at most $L$, and so the annular cover of the component of $\partial \widetilde M_\infty$ containing this curve has modulus at least $\frac{\pi}{L}$. Consequently, the interior of the half-open annulus (which is half of this annular cover) has modulus at least $\frac{\pi}{2L}$.  But this annulus lifts to the annular cover of $\breve{\Sigma}$ to which $\gamma$ lifts, and hence this cover has modulus at least $\frac{\pi}{2L}$ by monotonicity of moduli of annuli.  This in turn implies that $\ell_{\sigma_\eta^\hyp}(\gamma) \leq 2L$, as required.
\end{proof}

\subsection{Interpolation and pants paths} \label{S:interpolation}

Two triangulations ${\mathcal T}$ and ${\mathcal T}'$ differ by a {\em flip move} if there are edges $e$ of ${\mathcal T}$ and $e'$ of ${\mathcal T}'$ so that $e$ and $e'$ intersect transversely in a single point and their complements in the $1$--skeleta agree:
\[ 
    {\mathcal T}^{(1)}-e = {\mathcal T}'^{(1)}-e'.
\]
In this case, let ${\mathcal T}*{\mathcal T}'$ be the triangulation obtained from ${\mathcal T}$ by adding a vertex at $e \cap e'$, subdividing each of $e$ and $e'$ and adding the subdivided $e'$ to the $1$--skeleton.  This is the ``minimal common subdivision" of ${\mathcal T}$ and ${\mathcal T}'$ which is illustrated in \Cref{fig:min-common-subdivision}.

\begin{figure}[htb!]
\centering
\def\svgwidth{2.5in}
\begin{tikzpicture}
\node at (0,0) {\includegraphics[width = 4in]{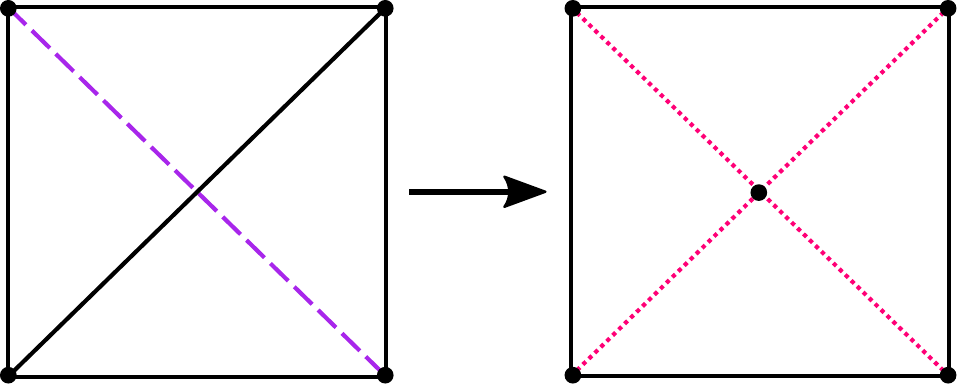}};
\node at (-1.9,1.45) {$e$};
\node at (-4,1.5) {$e'$};
\end{tikzpicture}
\caption{The minimal common subdivision ${\mathcal T} *{\mathcal T}'$ (shown on the right) of ${\mathcal T}$ and ${\mathcal T}'$ (shown on the left) where $e \subset {\mathcal T}$ is represented by the solid black diagonal edge and $e' \subset {\mathcal T}'$ is represented by the dashed purple diagonal edge. The four pink dotted edges in ${\mathcal T}*{\mathcal T}'$ are the new edges introduced by subdividing $e$ and $e'$}
\label{fig:min-common-subdivision}

\end{figure}

Given $\eta \colon (\Sigma,{\mathcal T}) \to \widetilde M_\infty$ in $\SH_{k_0}$ and a triangulation ${\mathcal T}'$ differing from ${\mathcal T}$ by a flip, we may reparameterize $\eta$ by precomposing with a homeomorphism isotopic to the identity so that it is also a $(k_0+1)$--simplicial hyperbolic surface
\[ 
    \eta \colon (\Sigma,{\mathcal T}*{\mathcal T}') \to \widetilde M_\infty.
\] 
The proof of \cite[Lemma~5.3]{Canary.CoveringTheorem} can be applied to prove the following.

\begin{lemma}\label{lemma.Canary.Family} Suppose ${\mathcal T}$ and ${\mathcal T}'$ differ by a flip and that $\eta \colon (\Sigma,{\mathcal T}) \to \widetilde M_\infty$ and $\eta' \colon (\Sigma,{\mathcal T}') \to \widetilde M_\infty$ lie in $\SH_{k_0}$.
Then there is a $1$--parameter family
\[ 
    \{
        \eta_t \colon (\Sigma,{\mathcal T}*{\mathcal T}') \to \widetilde M_\infty
        \mid 
        t\in [0,1]
    \} \subset \SH_{k_0+1}
\]
such that $\eta_0 = \eta$ and $\eta_1 = \eta'$, up to reparameterization by homeomorphisms isotopic to the identity by an isotopy that is stationary on the boundary. \qed
\end{lemma}
Rather than repeat all the details of the proof as in \cite{Canary.CoveringTheorem}, we explain the basic idea.  By further reparameterization if necessary, the two maps $\eta$ and $\eta'$ will agree outside the ``square" with diagonals $e$ and $e'$, and the interpolation takes place entirely within this square.  Lifting the restrictions of $\eta$ and $\eta'$ to the square, the two original triangles of $\mathcal T$ in this square, together with the two triangles of $\mathcal T'$ define a tetrahedron in $\mathbb H^3$.  Connecting the image of $e$ via the lift of $\eta$ to the image of $e'$ via the lift of $\eta'$ by a geodesic segment, the one-parameter family is essentially obtained by ``sliding" the new vertex along this geodesic, and then projecting back to $\widetilde M_\infty$.

We will also need the following result; see e.g.~Hatcher \cite{Hatcher1991}.

\begin{lemma}\label{lemma.Hatcher.Path}
Let ${\mathcal T}$ and ${\mathcal T}'$ be two triangulations of $\Sigma$ with $k_0$ vertices, one on each component of $\partial \Sigma$. Then there is a sequence ${\mathcal T}_0 = {\mathcal T}, {\mathcal T}_1, \ldots , {\mathcal T}_m$ of triangulations, each differing from the previous one by a flip so that $\mathcal T_m = \mathcal T'$, up to isotopy which is stationary on $\partial \Sigma$. 
\qed
\end{lemma}

\noindent We will use these two lemmas together with \Cref{L:realize simplicial} to prove the following corollary.

\begin{corollary} \label{C:flip interpolation}
Suppose 
 $\eta \colon (\Sigma,{\mathcal T}) \to \widetilde M_\infty$
 and 
 $\eta' \colon (\Sigma,{\mathcal T}') \to \widetilde M_\infty$ 
are two $k_0$--simplicial hyperbolic surfaces.
Then there is a $1$--parameter family
\[ 
    \{ 
        \eta_t \colon (\Sigma,{\mathcal T}_t) \to \widetilde M_\infty
        \mid t\in [0,1]\} 
        \subset \SH
\]
such that $\eta_0 = \eta$ and $\eta_1 = \eta'$.
\end{corollary}

\begin{proof}
Let $\mathcal T = \mathcal T_0,\ldots, \mathcal T_m = \mathcal T'$ be the sequence of flips from \Cref{lemma.Hatcher.Path}.  By \Cref{L:realize simplicial}, there is a sequence of simplicial hyperbolic surfaces
\[ \{\eta_j \colon (\Sigma,\mathcal T_j) \to \widetilde M_\infty\}_{j=0}^m\]
and $\eta_0 = \eta$, $\eta_m = \eta'$, up to reparameterization (isotopic to the identity by an isotopy that is stationary on the boundary).  By \Cref{lemma.Canary.Family}, for each $j=1,\ldots,m$, we can interpolate between $\eta_{j-1}$ and $\eta_j$ by a one-parameter family of simplicial hyperbolic surfaces.  Concatenating these one-parameter families, and reparameterizing by precomposing by isotopies between these families whenever necessary, produces the required 1-parameter family from $\eta$ to $\eta'$, as required.
\end{proof}

We are now able to define a continuous path in Teichm\"uller space given by the hyperbolic structures obtained from uniformization of the 1-parameter family of simplicial hyperbolic surfaces given to us in \Cref{C:flip interpolation}.

\begin{lemma} \label{L:hyperbolic path}
Given the family $\{\eta_t \colon (\Sigma,{\mathcal T}_t) \to \widetilde M_\infty\mid t\in [0,1]\} \subset \SH$
from \Cref{C:flip interpolation}, the map $[0,1] \to \Teich(\breve{\Sigma})$, given by $t \mapsto \sigma_{\eta_t}^\hyp$ defines a (continuous) path in $\Teich (\breve \Sigma)$.
\end{lemma}
\begin{proof}
The family in \Cref{lemma.Canary.Family} defines a continuous path in $\Teich(\breve \Sigma)$ since the shapes of the hyperbolic triangles in the interpolation vary continuously; see Section 5 of \cite{Canary.CoveringTheorem}.  The terminal point of the path from $\eta_{j-1}$ to $\eta_j$ and the initial point of the path from $\eta_j$ to $\eta_{j+1}$ differ by reparameterization by a homeomorphism isotopic to the identity.  This isotopy defines a constant path in $\Teich(\breve \Sigma)$ since the cone metrics are all obtained by pulling back the same cone metric by the homeomorphisms throughout the isotopy.  Therefore, the paths can be concatenated to produces a path from $\sigma_{\eta}^\hyp$ to $\sigma_{\eta'}^\hyp$, as required.
\end{proof}

Recall that $P_\alpha,P_\omega$ are pants decompositions on $\Sigma$ so that with respect to $\sigma_\alpha$ and $\sigma_\omega$, the lengths of each component of $P_\alpha$ and $P_\omega$, respectively are at most $L = L(\chi(f), \xi(f))$, from \Cref{L:bounded pants decomp}.  The next lemma says we can find simplicial hyperbolic surfaces whose cone metrics have hyperbolic uniformizations in which $P_\alpha$ and $P_\omega$ are also bounded length.

\begin{lemma} \label{L:nearly pleated}
There exists $L_1 = L_1(\chi(f), \xi(f))$ and a pair of simplicial hyperbolic surfaces
$\eta_{\alpha} \colon (\Sigma,{\mathcal T}_{\alpha}) \to \widetilde M_\infty$
and
$
    \eta_{\omega} \colon (\Sigma,{\mathcal T}_{\omega}) \to \widetilde M_\infty 
$
in $\SH_{k_0}$
such that each component of $P_{\alpha}$ and $P_{\omega}$ have length at most $L_1$ in $\sigma_{{\mathcal T}_{\alpha}}^\hyp$ and $\sigma_{{\mathcal T}_{\omega}}^\hyp$, respectively.
\end{lemma}

\begin{remark} We will ultimately use an interpolation through simplicial hyperbolic surfaces to find a path in the pants graph between $P_\alpha$ and $P_\omega$ via \Cref{L:hyperbolic path}.
In the finite-type case \cite{Brock-mappingtorus-vol}, Brock similarly constructs such a path, though in his situation, the initial and terminal pants decompositions are {\em defined} from a simplicial hyperbolic surface and its image under a power of the monodromy.  In our case, $\Sigma$ is not invariant by any nontrivial power of $f$. However, \Cref{L:nearly pleated} allows us to choose simplicial hyperbolic surfaces that are adapted to our {\em existing} pants decompositions $P_\alpha$ and $P_\omega$ (as opposed to choosing these pants decompositions from the simplicial hyperbolic surfaces themselves). Our construction in \Cref{L:nearly pleated} is guided by the minimally well-pleated surfaces $\phi_\alpha$ and $\phi_\omega$ that realize $P_\alpha$ and $P_\omega$, respectively.
\end{remark}

\begin{proof}[Proof of \Cref{L:nearly pleated}] We carry out the proof for $P_\alpha$ with the one for $P_\omega$ being identical.
Recall that $P_\alpha$ is the pants decomposition in $\Sigma$ so that the pleating locus of $\phi_\alpha$ is contained in a pants lamination $\lambda$ containing $P_\alpha$.  Choose a triangulation $\mathcal T$ of $\Sigma$ so that the arcs of intersection of the $1$--skeleton with each pair of pants connect every pair of distinct boundary curves in those pants.  To do this, we can first find a collection of pairwise disjoint essential arcs with this property, slide the endpoints to lie on the $k_0$ vertices, then extend to a triangulation.

For each closed curve $\delta$ in $P_\alpha \cup \partial \Sigma$, the non-compact leaves spiraling in toward $\delta$ spiral to the left or to the right (from both sides when $\delta \subset P_\alpha$, by assumption; see \Cref{section.pleated}).  Let $D \colon \Sigma \to \Sigma$ be the multitwist obtained by applying a right-handed Dehn twist around those $\delta \subset P_\alpha \cup \partial \Sigma$ for which the spiraling is to the right, and a left-handed Dehn twist around those $\delta$ for which the spiraling is to the left, and set $\mathcal T_j = D^j(\mathcal T)$, for all integers $j \geq 0$.  We also let $\eta_j \colon (\Sigma,\mathcal T_j) \to \widetilde M_\infty$ be the simplicial hyperbolic surface given by \Cref{L:realize simplicial}.  

Consider the sequence of triangulations $\mathcal T_j$ on $\Sigma$, straightened to be geodesic with respect to $\sigma_\alpha$.  Lift $\phi_\alpha$ to the universal covers $\tilde \phi_\alpha \colon \widetilde \Sigma \to \widetilde M \subset \mathbb H^3$.  We write $\widetilde{\mathcal T}_j$ for the lifted triangulation.
\begin{claim} \label{claim:nearly geodesic}
Given $\epsilon > 0$, there exists $J > 0$ so that for all $j \geq J$ and every edge $e$ of $\widetilde{\mathcal T}_j$,  $\tilde \phi_\alpha\big|_e$ is a $(1+\epsilon,\epsilon)$--quasi-geodesic.
\end{claim}
\begin{proof}
First, we observe that as $j$ tends to infinity, we have Hausdorff convergence of the $1$--skeleta, $\mathcal T_j^{(1)} \to \lambda$. This is because all angles of intersection with $P_\alpha \cup \partial \Sigma$ tend to zero, so the limit is a lamination containing $P_\alpha \cup \partial \Sigma$. See \Cref{fig:spinning}. Further, in any pair of pants, our original choice of $\mathcal T$ guarantees that there are non-compact leaves that spiral between any two boundary components.  Finally, our choice of $D$ ensures that the spiraling toward all the curves is in the correct direction.  These conditions uniquely determine the pants lamination $\lambda$.
If $\tilde \lambda$ is the lifted lamination to $\widetilde \Sigma$, then we also have $\widetilde{\mathcal T}_j^{(1)} \to \tilde \lambda$ as $j \to \infty$.
Thus, given $c, r > 0$ there exists $J >0$ so that for all $j > J$, every edge $e$ of $\widetilde{\mathcal T}_j$, and every segment $\delta \subset e$ of length at most $r$ and there is a segment $\delta'$ of a leaf of $\lambda$ such that the endpoints of $\delta'$ and the endpoints of $\delta$ are at most $c$ apart. 

Since $\tilde \phi_\alpha$ maps each leaf of $\tilde \lambda$ to a geodesic and is $1$--Lipschitz, it follows that the endpoints of $\tilde \phi_\alpha(\delta)$ are within $c$ of the endpoints of the geodesic $\tilde\phi_\alpha(\delta')$.  Therefore, if $x,y$ are the endpoints of $\delta$, we have
\[ d(\tilde \phi_\alpha(x),\tilde \phi_\alpha(y)) \geq \ell(\tilde \phi_\alpha(\delta'))-2c = \ell(\delta')-2c \geq \ell(\delta)-4 c = d(x,y) - 4 c.\]
Since $\tilde \phi_\alpha\big|_e$ is $1$--Lipschitz, it follows that this path is an $r$--local, $(1,4c)$--quasi-geodesic (\textit{i.e.}~every segment of length less than $r$ is a $(1, 4c)$--quasi-geodesic).  Using hyperbolic geometry, we can find $r$ sufficiently large and $c$ sufficiently small so that such a path is also $(1+\epsilon,\epsilon)$--quasi-geodesic.  Specifically, take $r=4$, and consider consecutive points $x_0,x_1,\ldots,x_n$ along $e$ with $1 \leq d(x_j,x_{j+1}) \leq 2$.  Taking $c$ sufficiently small, the angle at $\tilde \phi_\alpha(x_j)$ between geodesic segments $[\tilde \phi_\alpha(x_{j-1}),\tilde \phi_\alpha(x_j)]$ and $[\tilde \phi_\alpha(x_j),\tilde \phi_\alpha(x_{j+1})]$ in $\widetilde M \subset \mathbb H^3$ can be made arbitrarily close to $\pi$ (depending on $c$), and then we can apply \cite[Theorem~4.2.10]{CanaryEpsteinGreen} to see that the concatenation of geodesic segments is arbitrarily close to the geodesic.  Then $\tilde \phi_\alpha\big|_e$ is also as close to a geodesic as we like, and we can promote the local quasi-geodesic to a global quasi-geodesic.
Taking $J$ large enough to find such an $r$ and $c$, completes the proof of the claim.
\end{proof}

\begin{figure}[htb!]
\centering
\def\svgwidth{3in}
\includegraphics[width = 4in]{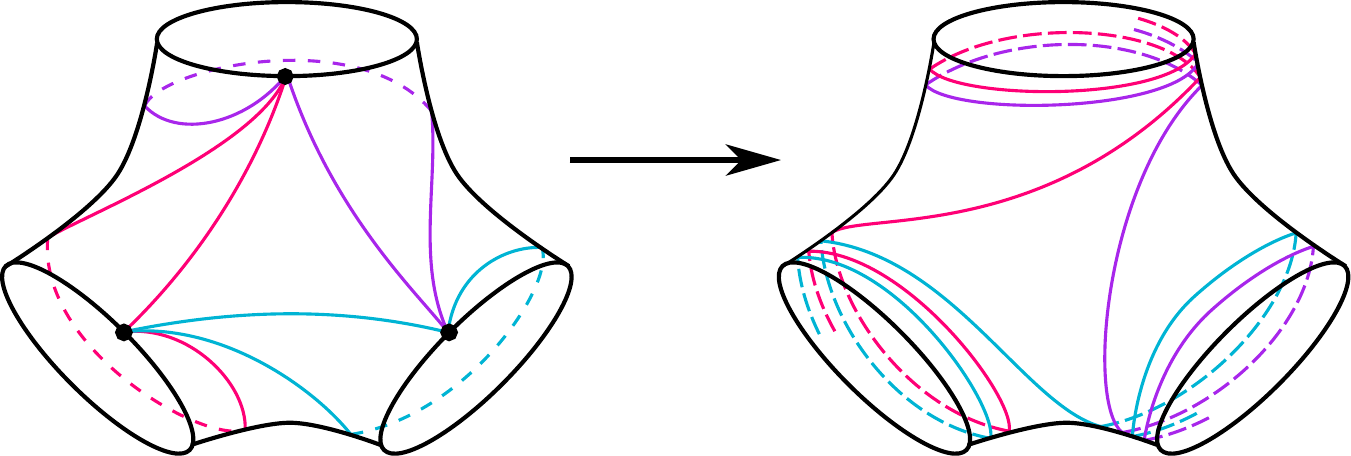}
\caption{By ``spinning" the triangulation on the left around each boundary component---more precisely, applying powers of Dehn twists about curves parallel to the boundary---we obtain the pants lamination in the limit, as shown on the right. Note that this is similar to (but distinct from) the spinning described in Thurston's notes \cite[Section 8.7]{TNotes} since we are applying powers of Dehn twists rather than pushing a point along a geodesic.}
\label{fig:spinning}
\end{figure}

\begin{claim} \label{Claim:nearly pleated}
Given $\epsilon >0$, there exists $J > 0$ so that for all $j \geq J$:
\begin{enumerate}
    \item For every component $\gamma \subset P_\alpha$, we have $\ell_{\sigma_{\eta_j}}(\gamma) \leq \ell_{\sigma_\alpha}(\gamma)+\epsilon$, and
    \item all cone angles in the boundary of $\Sigma$ for $\mathcal T_j$ are at most $\pi+\epsilon$.
\end{enumerate}
\end{claim}
Note that because $\sigma_\alpha$ maps any component $\gamma \subset P_\alpha$ to a geodesic, and since $\eta_j$ is $1$--Lipschitz with respect to $\sigma_{\eta_j}$, we have
$\ell(\phi_\alpha(\gamma)) = \ell_{\sigma_\alpha}(\gamma) \leq \ell_{\sigma_{\eta_j}}(\gamma)$, so part (1) in the claim really says that $\ell_{\sigma_\alpha}(\gamma)$ and $\ell_{\sigma_{\eta_j}}(\gamma)$ are nearly the same.
\begin{proof}
Let $m$ be the maximum geometric intersection number between (the union of arcs in) $\mathcal T^{(1)}$ and any component of $P_\alpha$.  Since $\mathcal T^{(1)}_j = D^j(\mathcal T^{(1)})$, this maximum, $m$, also bounds the geometric intersection number between $\mathcal T^{(1)}_j$ and any component of $P_\alpha$.

Now, let $\tilde \eta_j \colon \widetilde \Sigma \to \widetilde M \subset \mathbb H^3$ be the lift of the simplicial hyperbolic surface $\eta_j \colon (\Sigma,\mathcal T_j) \to \widetilde M_\infty$.
From \Cref{claim:nearly geodesic}, we deduce that, for $j$ sufficiently large, and for every edge $e$ of $\widetilde{\mathcal T}_j$, $\tilde\eta_j\big|_e$ and $\tilde \phi_\alpha|_e$ can be made arbitrarily close (depending on $j$).

For any component $\gamma \subset P_\alpha$ pick a lift $\widetilde \gamma \subset \widetilde \Sigma$ and a segment $\bar \gamma$ that serves as a fundamental domain for the action of $\langle g \rangle$, the stabilizer of $\widetilde \gamma$ in $\pi_1\Sigma$.  Then, for any $j$ consider the set of (at most $m$) edges of $\widetilde{\mathcal T}_j$ that cross $\widetilde \gamma$ at a point of $\bar \gamma$.  For $j$ sufficiently large, the $\tilde \phi_\alpha$-image of these edges are arbitrarily close to the geodesic $\tilde \phi_\alpha(\widetilde \gamma)$ on arbitrarily long segments (depending on $j$), thus the same is true for the $\tilde \eta_j$ images of these edges.  

Next, pick the edge $e \subset \widetilde{\mathcal T}_j$ intersecting $\bar \gamma$ as close to the initial point of $\bar \gamma$ as possible and assume that $j$ is sufficiently large so that all the edges between $e$ and $g \cdot e$ (ordered by the intersection with $\widetilde \gamma$) are mapped within $\frac{\epsilon}{2(m+1)}$ of $\tilde \phi_\alpha(\widetilde \gamma)$ on segments of length at least $\ell_{\sigma_\alpha}(\gamma) + \epsilon$ centered at the points of intersection with $\widetilde \gamma$.

\begin{figure}[htb!]
\begin{tikzpicture}
\node at (0,0) {\includegraphics[width=12cm]{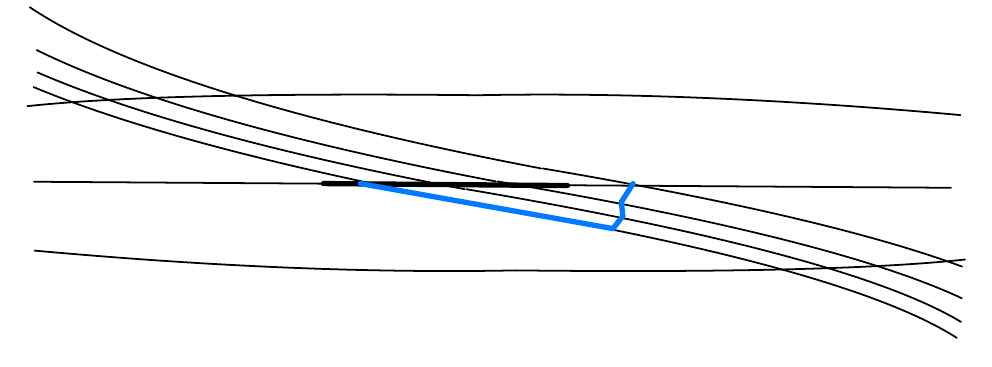}};

\draw[->,>=stealth] (-2.7,-.5) -- (-2,-.05);
\draw[->,>=stealth] (1.7,.2) -- (3.7,.2);
\node at (5.1,-.5) {$g \cdot e$};
\node at (-4.4,.4) {$e$};
\node at (-3,-.4) {$\bar \gamma$};

\node at (2.7,.5) {$g$};
\node at (-5.8,.1) {$\widetilde \gamma$};
\end{tikzpicture}
\caption{The edges of $\widetilde{\mathcal T}_j$ through $\widetilde \gamma$ between $e$ and $g \cdot e$ are mapped close to $\tilde \phi_\alpha(\widetilde \gamma)$ for a long time by $\eta_j$.  The blue segment connects $e \cap \widetilde \gamma$ to $g \cdot e \cap \widetilde \gamma$, and projects to a closed loop in $\Sigma$.} \label{F:edges near geodesic}
\end{figure}
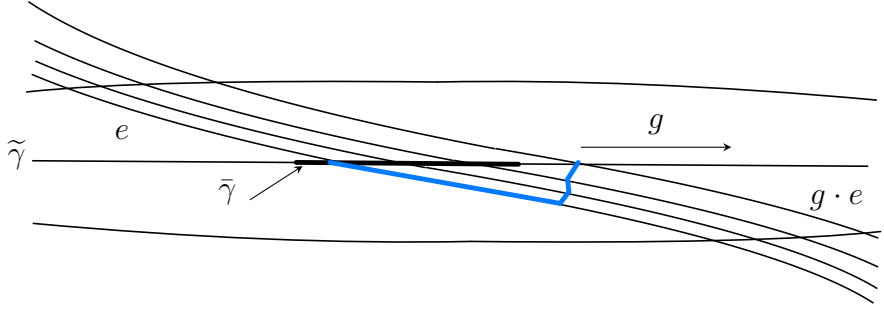

From this we can construct a path in $\widetilde \Sigma$ from $e \cap \widetilde \gamma$ to $g \cdot e \cap \widetilde \gamma$ built from a segment of $e$ of length at most $\ell_{\sigma_\alpha}(\gamma) + \frac{\epsilon}{2(m+1)}$ and at most $m$ short segments of length at most $2 \frac{\epsilon}{2(m+1)}$ between consecutive edges of $\widetilde{\mathcal T}_j$.  Thus the $\sigma_{\eta_j}$--length of this path is at most
\[ \ell_{\sigma_\alpha}(\gamma) + \frac{\epsilon}{2(m+1)} + m\left( \frac{\epsilon}{(m+1)}\right) < \ell_{\sigma_\alpha}(\gamma)+\epsilon.\]
See \Cref{F:edges near geodesic}.
This path then projects to a loop in $\Sigma$ homotopic to $\gamma$ of $\sigma_{\eta_j}$--length at most $\ell_{\sigma_\alpha}(\gamma)+\epsilon$. This proves part (1) of the claim.

For part (2), we similarly observe that for any vertex $v$, any edge $e$ adjacent to $v$ has $\eta_j$--image that stays arbitrarily close to the geodesic image of the boundary component $\eta_j(\gamma)$ for an arbitrarily  long time (depending on $j$).  Now the boundary component $\gamma$ gives two adjacencies to $v$ pointing in opposite directions, and all other edges adjacent to $v$ have $\eta_j$--image making arbitrarily small angle with the $\eta_j$--image of exactly one of these.  It follows that there is one angle equal to at most $\pi$, and at most $m$ angles that can be made arbitrarily small, depending on $j$.  It follows that for $j$ sufficiently large, the angle can be made at most $\pi+\epsilon$, proving part (2).
\end{proof}

To complete the proof, we now observe that since the $\sigma_{\eta_j}$--length of each boundary curve is the same as its length with respect to $\sigma_\alpha$, it is bounded by $L$.  Further, since the cone angle is less than $\pi+\epsilon$, then provided $\epsilon < \pi$, we can enlarge $\Sigma$ to an open surface, $\Sigma \subset \Sigma^\circ$, and extend $\sigma_\alpha$ to a complete {\em non-singular} hyperbolic metric.  The convex core $\bar \Sigma$ of $\Sigma^\circ$ will contain $\Sigma$; see \Cref{F:enlarge and core}. 

\begin{figure}[htb!]
\def\svgwidth{3in}
\begingroup%
  \makeatletter%
  \providecommand\color[2][]{%
    \errmessage{(Inkscape) Color is used for the text in Inkscape, but the package 'color.sty' is not loaded}%
    \renewcommand\color[2][]{}%
  }%
  \providecommand\transparent[1]{%
    \errmessage{(Inkscape) Transparency is used (non-zero) for the text in Inkscape, but the package 'transparent.sty' is not loaded}%
    \renewcommand\transparent[1]{}%
  }%
  \providecommand\rotatebox[2]{#2}%
  \newcommand*\fsize{\dimexpr\f@size pt\relax}%
  \newcommand*\lineheight[1]{\fontsize{\fsize}{#1\fsize}\selectfont}%
  \ifx\svgwidth\undefined%
    \setlength{\unitlength}{696.50793402bp}%
    \ifx\svgscale\undefined%
      \relax%
    \else%
      \setlength{\unitlength}{\unitlength * \real{\svgscale}}%
    \fi%
  \else%
    \setlength{\unitlength}{\svgwidth}%
  \fi%
  \global\let\svgwidth\undefined%
  \global\let\svgscale\undefined%
  \makeatother%
  \begin{picture}(1,0.68235955)%
    \lineheight{1}%
    \setlength\tabcolsep{0pt}%
    \put(0,0){\includegraphics[width=\unitlength,page=1]{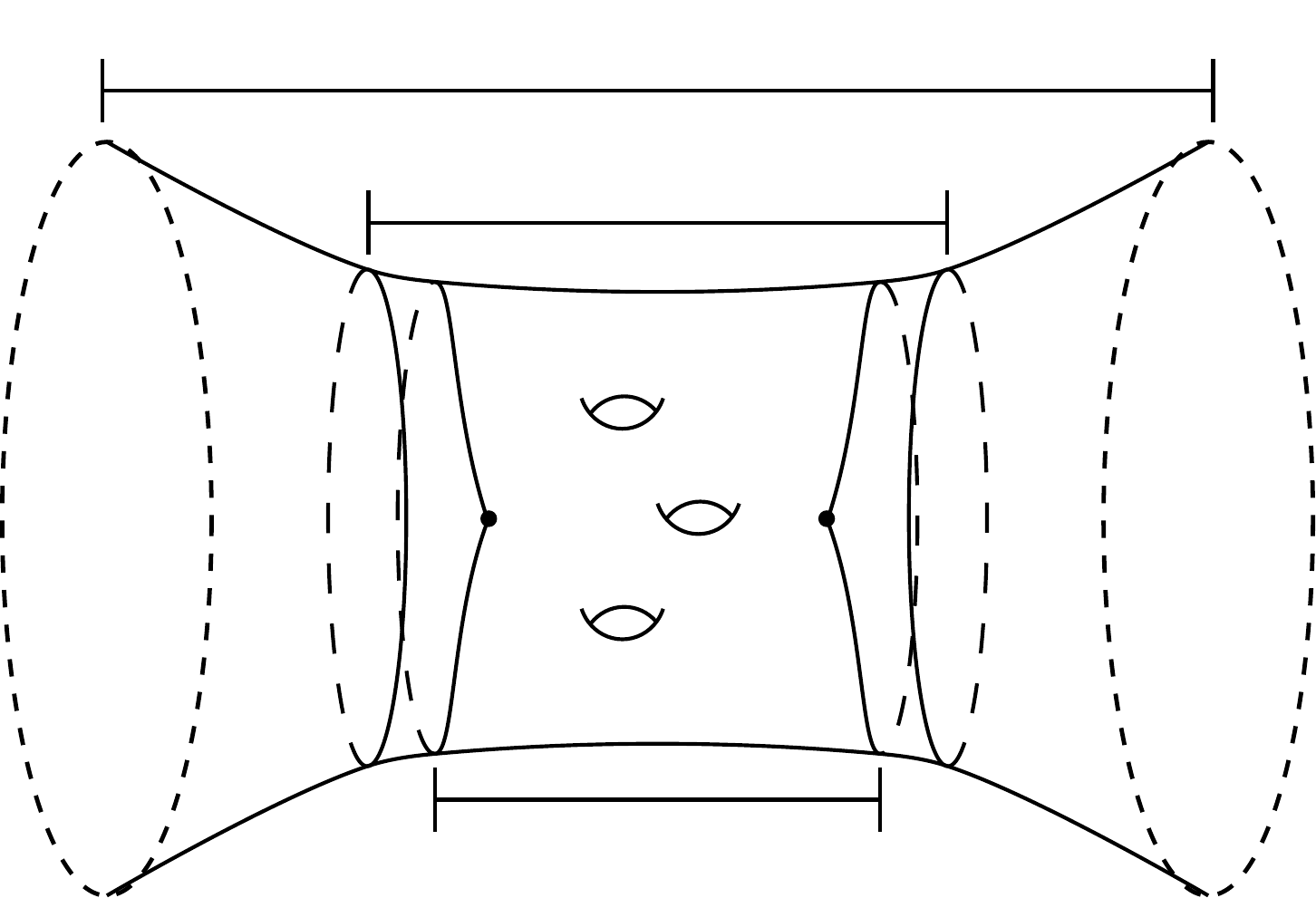}}%
    \put(0.49981492,0.62932713){\color[rgb]{0,0,0}\makebox(0,0)[t]{\lineheight{1.25}\smash{\begin{tabular}[t]{c}$\Sigma^{\circ}$\end{tabular}}}}%
    \put(0.49981492,0.52955705){\color[rgb]{0,0,0}\makebox(0,0)[t]{\lineheight{1.25}\smash{\begin{tabular}[t]{c}$\bar \Sigma$\end{tabular}}}}%
    \put(0.49981492,0.02142558){\color[rgb]{0,0,0}\makebox(0,0)[t]{\lineheight{1.25}\smash{\begin{tabular}[t]{c}$\Sigma$\end{tabular}}}}%
  \end{picture}%
\endgroup%

\caption{Enlarging $\Sigma$ to $\Sigma^\circ$, with convex core $\bar \Sigma$.} \label{F:enlarge and core}
\end{figure}
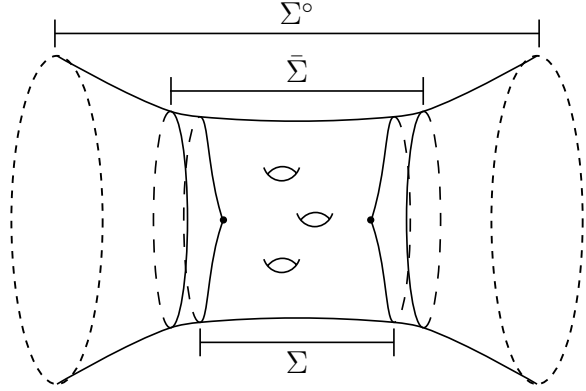

Given any $\epsilon' > 0$, there exists $\epsilon > 0$ so that if $\gamma \subset \partial \Sigma$ is any component and if the cone angle at the vertex on $\gamma$ has angle less than $\pi + \epsilon$, then $\bar \Sigma \subset N_{\epsilon'}(\Sigma)$ on $\Sigma^\circ$ with respect to $\sigma_{\eta_j}$, for any $j \geq J$ as in \Cref{Claim:nearly pleated}.  We note that the choice of $\epsilon$ depends not only on $\epsilon'$, but also on a lower bound on $\ell_{\sigma_{\eta_j}}(\gamma) = \ell_{\sigma_\alpha}(\gamma)$.

For any $0 < \epsilon < 1$, $j \geq J$ as in \Cref{Claim:nearly pleated}, and any component $\gamma \subset P_\alpha$, $\ell_{\sigma_{\eta_j}}(\gamma) < L + \epsilon < 1$.  By the collar lemma, there exists $w = w(L+1)$ such that the $w$--neighborhood of $\gamma$ is an embedded annulus in $\bar \Sigma$.  Thus, if $\epsilon' < \frac{w}2$, and $\epsilon$ is chosen as in the previous paragraph, then the collar of width $w/2$ is entirely contained in $\Sigma \subset \breve{\Sigma}$.  This gives a lower bound on the $\sigma_{\eta_j}$--moduli of $\gamma$, and hence an upper bound $L_1$ on $\ell_{\sigma_{\eta_j}^\hyp}(\gamma)$ depending on $L$ and $w$, which in turn both depend only on  $\chi(f)$ and $\xi(f)$, as required.
\end{proof}

\begin{remark} We could also find the sequence of simplicial hyperbolic surfaces limiting to $\phi_\alpha \colon \Sigma \to \widetilde M_\infty$ (and likewise for $\phi_\omega$) by adding extra vertices to each pants curve, triangulating each pair of pants, and then appealing to Thurston's original construction of spinning from \cite[Section 8.7]{TNotes}.  The interpolation in that case would not be between triangulations of $\Sigma$ with $k_0$ vertices as in \Cref{lemma.Hatcher.Path}, but instead, for example, could be carried out via a sequence of pants decompositions with marked points on the boundary and the pants curves.  This alternate approach would allow for more general pants laminations which are not required to satisfy the spiraling constraints from \Cref{section.pleated}.
\end{remark}

\subsection{Pants distances and bounded length curves} \label{S:pants distance bounded length} 

Write $\mathcal T_{2L}(\breve{\Sigma})$ to denote the subspace of $\mathcal T(\breve{\Sigma})$ consisting of complete hyperbolic structures for which the (geodesic representatives of) components of $\partial \Sigma$ have length at most $2L$.  Set
\begin{equation} \label{Eq:final complexity} \Xi = \Xi(f)= \tfrac32\left(\xi(f) + |\chi(f)|\right).\end{equation}
From \Cref{Eq:Big Sigma Euler} we have that $\xi(\Sigma) \leq \frac32 |\chi(\Sigma)| \leq \Xi$.
Now set
\[ 
    L_2 = L_2(\chi(f),\xi(f))= \max\{L_1,L_B(\Xi,2L)\},
\]
where $L_B$ is as in \Cref{thm:relative-bers}.  

For any pants decomposition $P_0$ of $\Sigma$, let
\[ 
    V(P_0,L_2) = \{\sigma \in \mathcal T_{2L}(\Sigma) \mid \ell_\sigma(\gamma) < L_2 \mbox{ for each component } \gamma \subset P_0 \}.
\]
By \Cref{thm:relative-bers}, the set $\{V(P_0,L_2)\}_{P_0 \in \CP(\Sigma)}$ is an open cover of $\mathcal T_{2L}(\Sigma)$.

Now let 
\[ 
    \{(\eta_t \colon (\Sigma,{\mathcal T}_t) \to \widetilde M_\infty)\mid t\in [0,1]\} \subset \SH
\]
be the path of simplicial hyperbolic surfaces from \Cref{L:hyperbolic path} connecting the $k_0$--simplicial hyperbolic surfaces $\eta_{\alpha} \colon (\Sigma,{\mathcal T}_{\alpha}) \to \widetilde M_\infty$ to $\eta_{\omega} \colon (\Sigma,{\mathcal T}_{\omega}) \to \widetilde M_\infty$ from \Cref{L:nearly pleated}.  By compactness (and \Cref{L:nearly pleated}), there is a partition
\[ 
    0 = t_0 < t_1 < \cdots < t_n = 1 
\]
and pants decompositions $P_j$, for $j =1,\ldots,n$, so that $P_1 = P_{\alpha}$, $P_n = P_{\omega}$, and
\[ 
    \sigma_t \in V(P_j,L_2) 
\]
for all $t \in [t_{j-1},t_j]$ and $j = 1,\ldots,n$.

Since $P_j$ and $P_{j+1}$ both have length at most $L_2$ with respect to $\sigma_{\eta_{t_j}}^\hyp$, Lemma~3.3 of \cite{Brock-convex-vol} implies that $d_{\CP}(P_j,P_{j+1}) \leq \kappa = \kappa(L_2,\Xi)$.  

Write $\mathcal S(P_1,\ldots,P_n)$ to denote the union of all curves in all pants decompositions $P_1,\ldots,P_n$. The next result is also due to Brock {\cite[Lemma~4.3]{Brock-convex-vol}}.

\begin{lemma} \label{L:Brock sequence distance}
 There exists $\mathcal K$ depending on $\kappa > 0$ and $\xi(\Sigma)$ with the following property.
Let $P_1,\ldots,P_n$ be a sequence of pants decompositions of $\Sigma$ such that 
\[
    d_{\CP}(P_j,P_{j+1}) \leq \kappa,
\]
for all $j = 1,\ldots,n-1$.
Then
\[ 
    d_{\CP}(P_1,P_n) \leq \mathcal K |\mathcal S(P_1,\ldots,P_n)|.
\]
\end{lemma}

For us, the key application of this lemma is the following. 

\begin{corollary} \label{C:cal K bound} Let $P_{\alpha}= P_1,\ldots,P_n = P_{\omega}$ be as above and $\mathcal K$ as in Lemma~\ref{L:Brock sequence distance} (which depends only on $\xi(f)$ and $\chi(f)$).  Then
\[ 
    d_{\CP}(P_{\alpha},P_{\omega}) \leq \mathcal K |\mathcal S(P_1,\ldots,P_n)|.
\]
\end{corollary}

Finally, from \Cref{T:Ahlfors} we have the following.
\begin{lemma} \label{L:lots of bounded curves}
Let $P_{\alpha}= P_1,\ldots,P_n = P_{\omega}$ be as above.  Then the geodesic representative in $\widetilde M_\infty$ of each curve in $\mathcal S(P_1,\ldots,P_n)$ has length bounded above by $L_2$. \qed
\end{lemma}

\section{Bounding volume}
\label{section.volume}

We are almost ready to prove the main theorem. We will need one more result, again due to Brock \cite[Lemma~4.8]{Brock-convex-vol}.  Suppose $M$ is compact, convex, hyperbolic $3$-manifold, $\mathcal L > 0$, and let $\mathcal G_{\mathcal L}(M)$ denote the set of closed geodesics in $M$ with length less than $\mathcal L$. 

\begin{proposition} \label{P:Brock lower volume}
Given $\mathcal L > 0$ greater than the Bers constant for closed surfaces of complexity $\xi$, there is a constant $\mathcal V = \mathcal V(\xi,\mathcal L) > 0$ with the following property.  Given a compact hyperbolic $3$-manifold $M$ with totally geodesic boundary and $\xi(\partial M) \leq \xi$, then
\[ 
    \Vol(M) \geq \mathcal V |\mathcal G_{\mathcal L}(M)|.
\]
\end{proposition}
\begin{remark}
Brock's statement in \cite{Brock-convex-vol} involves an additive error as well that depends on $\chi(\partial M)$ (and not $\mathcal L$).  However, it does not require $\mathcal L$ sufficiently large (in our statement, greater than the Bers constant).  Because we assume $M$ is acylindrical in this statement, there is a uniform lower bound to the volume (approximately 6.452...)~by a result of Kojima and Miyamoto \cite{KojimaMiyamoto}, and, because we have assumed $\mathcal L$ is greater than the Bers constant, $\mathcal G_{\mathcal L}(M)$ is nonempty.  Consequently, we can absorb the additive constant into the multiplicative one, arriving at the version that is most useful for our purposes.
\end{remark}

\begin{proof}[Proof of \Cref{T:main.theorem}]
Let $P_{\alpha} =P_1,\ldots, P_n = P_{\omega}$ be the sequence of pants decomposition constructed in \S\ref{S:pants distance bounded length} on $\Sigma$ and $\mathcal S(P_1,\ldots,P_n)$ the set of all curves in all pants decompositions $P_1,\ldots, P_n$.  By \Cref{L:lots of bounded curves},
\[
    \mathcal S(P_1,\ldots,P_n) \subset \mathcal G_{L_2}(\widetilde M_\infty).
\]

By \Cref{L:no-closed-curves}, setting $N = 2 \xi(\Sigma) \leq 2\Xi(f)$, we have that $\Sigma$ and $f^m(\Sigma)$ have no closed curves in common for all $m \geq N$.  In particular, no two curves in $\mathcal S(P_1,\ldots,P_n)$ differ by an element of $\langle f^N \rangle$.  Consequently, no two elements of $\mathcal S(P_1,\ldots,P_n)$ project to the same homotopy class in $\overline M_{f^N}$.  In particular, we have
\begin{equation} \label{E:lots of curves inject} |\mathcal S(P_1,\ldots,P_n)| \leq |\mathcal G_{L_2}(\overline M_{f^N})|.
\end{equation}

Now observe that
\begin{eqnarray*}
\tau(f) & \leq & d_{\CP}(P,f(P)) \leq  d_{\CP}(P_{\alpha},P_{\omega}) \leq  \mathcal K|\mathcal S(P_1,\ldots,P_n)|\\\\
    & \leq & \mathcal K |\mathcal G_{L_2}(\overline M_{f^N})| \leq
 \frac{\mathcal K}{\mathcal 
    V} \Vol(\overline M_{f^N}) = \frac{N \mathcal K}{\mathcal V} \Vol(\overline M_f).
\end{eqnarray*}
The first inequality is by definition.  The second follows from \Cref{E:reduction to finite type pants}.  The third is by \Cref{C:cal K bound}.  The fourth follows from \Cref{E:lots of curves inject}.  The fifth inequality comes from \Cref{P:Brock lower volume} (note that we must adjust $\mathcal V$ because of the power $N$, but this is also uniform depending on the capacity).  The final equality comes from the fact that $\overline M_{f^N}$ is an $N$--fold cover of $\overline M_f$. Since $N,\mathcal K, \mathcal V$ all depend only on $\xi(f)$ and $\chi(f)$, this completes the proof.
\end{proof}

\section{Bounded length invariant components}
\label{sec:bdd-length-inv-comps}

Any $f$--invariant component $\Omega \subset \mathcal P(S)$ determines a pants decomposition $P_\Omega$ of $\partial \overline M_f$ (see \cite{EndPeriodic1}).  More precisely, if $P \in \Omega$ is any pants decomposition representing a vertex in this component then, after identifying $\partial \overline M_f$ with the quotient $S_+ \sqcup S_- = (\mathcal U_+ \sqcup \mathcal U_-)/\langle f \rangle$, the preimage of $P_\Omega$ in $\mathcal U_+ \sqcup \mathcal U_-$ agrees with $P$ on neighborhoods of the attracting and repelling ends of $\mathcal U_+$ and $\mathcal U_-$, respectively.  Given an $f$--invariant component $\Omega \subset \mathcal P(S)$, the pants decomposition $P_\Omega$ can be constructed by first observing that for any $P \in \Omega$ there are good nesting neighborhoods $U_\pm$ so that $P$ defines a pants decomposition $P|_{U_\pm}$ of $U_\pm$ (in particular, $\partial U_\pm$ is a union of curves in $P$) and so that $f^{\pm 1}(P|_{U_\pm}) \subset P|_{U_{\pm}}$.  It follows that 
\[ \bigcup_{k=0}^{\infty} f^{\mp k}(P|_{U_\pm}) \] 
is a pants decomposition of $\mathcal U_\pm$ which is $\langle f \rangle$--invariant, and hence, descends to a pants decomposition $P_\Omega$ on $S_\pm$.

The construction of the pants decomposition $P$ in the proof of \Cref{L:bounded pants decomp} defines such an $f$--invariant component $\Omega_0 \subset \mathcal P(S)$, and can be explicitly described as follows.  The subsurface $\Delta_+ \subset \overline{U_+} \subset \mathcal U_+$ is a fundamental domain for the action of $\langle f \rangle$, and the chosen pants decomposition $P_+$ from that proof projects to the components of $P_{\Omega_0}$ contained in $\partial_+ \overline M_f$.   A similar statement is true for the components of $P_{\Omega_0}$ in $\partial_+ \overline M_f$.  We note that the components of $P_{\Omega_0}$ have uniformly bounded length, depending only on the capacity by \Cref{L:bounded pants decomp}.

Every $f$--invariant component $\Omega \subset \mathcal P(S)$ has its own translation length
\[ \tau_\Omega(f) = \inf_{P \in \Omega} \lim_{k \to \infty} \frac{d_{\mathcal P}(P,f^k(P))}{k}. \]
As was shown in \cite{EndPeriodic1}, for any strongly irreducible end-periodic homeomorphisms $f$, there is always a sequence of $f$--invariant components $\Omega_n$ so that $\tau_{\Omega_n}(f) \to \infty$ as $n \to \infty$.
On the other hand, by definition
\[ \tau(f) = \inf_{k,\Omega}\frac{\tau_\Omega(f^k)}k,\]
where the infimum is taken over all $k \geq 1$ and $f^k$--invariant components $\Omega$.

Thus there are two measures of efficiency for a component $\Omega \subset \mathcal P(S)$ with respect to a strongly irreducible end-periodic homeomorphism $f \colon S \to S$.  The first is that $P_\Omega$ has bounded length, which is a geometric condition in terms of the hyperbolic geometry of $\overline M_f$.  The second is purely topological/combinatorial, and is that  $\tau_\Omega(f)$ approximates $\tau(f)$.  The next result says that these can be achieved simultaneously.

\medskip

\noindent{\bf \Cref{T:bdd len compts}}
{\em \boundedlengththm }

\smallskip

\begin{proof}
We claim that the component $\Omega_0$ defined by $P$ from \Cref{L:bounded pants decomp} satisfies the conditions of the proposition.  Since $U_+ \cup U_-$ projects locally isometrically to $\partial \overline M_f$, and since the components of $P$ have length bounded by $L = L(\xi(f),\chi(f))$, the component of $P_{\Omega_0}$ are similarly bounded by $L$.

To see that $\tau_{\Omega_0}(f)$ is bounded by a uniform constant multiple of $\tau(f)$, we first observe that the proof of \Cref{T:main.theorem} in fact shows that
\[ \tau_{\Omega_0}(f) \leq \frac1{C_2} \Vol(\overline M_f), \]
where $C_2$ was explicitly shown to be given by $\frac{\mathcal V}{N \mathcal K}$.  On the other hand, $\Vol(\overline M_f) \leq V_\oct \tau(f)$ by the main result of \cite{EndPeriodic1}.  Therefore,
\[\tau_{\Omega_0}(f) \leq \frac{V_\oct}{C_2} \tau(f).\]
Setting $E = \max \left\{L,\frac{V_\oct}{C_2}\right\}$ proves the theorem.
\end{proof}

\subsection*{Acknowledgments}
Field and Loving were supported in part by NSF Mathematical Sciences Postdoctoral Research Fellowships. The authors were also supported in part by NSF grants DMS-1840190 and DMS-2103275 (Field),   DMS-1904130 and DMS-2202718 (Kent), DMS-2106419 and DMS-2305286 (Leininger), and DMS-2231286 (Loving). The authors would like to thank Yair Minsky for useful conversations, Heejoung Kim for her collaboration on the paper \cite{EndPeriodic1}, which motivated this one, and Hugo Parlier for pointing us towards his proof of the relative Bers lemma. They are grateful to Sam Taylor for a suggestion that simplified the proof of \Cref{T:bdd len compts} and the exposition more generally. Finally, they would like to thank Michael Landry, Chi Cheuk Tsang, Brandis Whitfield, and especially the anonymous referee for their careful reading and comments on various drafts of this paper.

\bibliographystyle{alpha}
  \bibliography{main}

\end{document}